%% file: directories.tex
\documentclass[times,final]{elsarticle}
\pdfoutput=1

\makeatletter
\@ifundefined{lhs2tex.lhs2tex.sty.read}%
  {\@namedef{lhs2tex.lhs2tex.sty.read}{}%
   \newcommand\SkipToFmtEnd{}%
   \newcommand\EndFmtInput{}%
   \long\def\SkipToFmtEnd#1\EndFmtInput{}%
  }\SkipToFmtEnd

\newcommand\ReadOnlyOnce[1]{\@ifundefined{#1}{\@namedef{#1}{}}\SkipToFmtEnd}
\usepackage{amstext}
\usepackage{amssymb}
\usepackage{stmaryrd}
\DeclareFontFamily{OT1}{cmtex}{}
\DeclareFontShape{OT1}{cmtex}{m}{n}
  {<5><6><7><8>cmtex8
   <9>cmtex9
   <10><10.95><12><14.4><17.28><20.74><24.88>cmtex10}{}
\DeclareFontShape{OT1}{cmtex}{m}{it}
  {<-> ssub * cmtt/m/it}{}

\DeclareFontShape{OT1}{cmtt}{bx}{n}
  {<5><6><7><8>cmtt8
   <9>cmbtt9
   <10><10.95><12><14.4><17.28><20.74><24.88>cmbtt10}{}
\DeclareFontShape{OT1}{cmtex}{bx}{n}
  {<-> ssub * cmtt/bx/n}{}

\newcommand{\Conid}[1]{\mathit{#1}}
\newcommand{\Varid}[1]{\mathit{#1}}
\newcommand{\anonymous}{\kern0.06em \vbox{\hrule\@width.5em}}

\newcommand{\bind}{\mathbin{>\!\!\!>\mkern-6.7mu=}}

\renewcommand{\leq}{\leqslant}

\usepackage{polytable}

\@ifundefined{mathindent}%
  {\newdimen\mathindent\mathindent\leftmargini}%
  {}%

\def\resethooks{%
  \global\let\SaveRestoreHook\empty
  \global\let\ColumnHook\empty}
\newcommand*{\savecolumns}[1][default]%
  {\g@addto@macro\SaveRestoreHook{\savecolumns[#1]}}
\newcommand*{\restorecolumns}[1][default]%
  {\g@addto@macro\SaveRestoreHook{\restorecolumns[#1]}}
\newcommand*{\aligncolumn}[2]%
  {\g@addto@macro\ColumnHook{\column{#1}{#2}}}

\resethooks

\newcommand{\onelinecommentchars}{\quad-{}- }
\newcommand{\commentbeginchars}{\enskip\{-}
\newcommand{\commentendchars}{-\}\enskip}

\newcommand{\visiblecomments}{%
  \let\onelinecomment=\onelinecommentchars
  \let\commentbegin=\commentbeginchars
  \let\commentend=\commentendchars}

\newcommand{\invisiblecomments}{%
  \let\onelinecomment=\empty
  \let\commentbegin=\empty
  \let\commentend=\empty}

\visiblecomments

\newlength{\blanklineskip}
\newcommand{\hsindent}[1]{\quad}
\let\hspre\empty
\let\hspost\empty

\EndFmtInput
\makeatother
\ReadOnlyOnce{polycode.fmt}%
\makeatletter

\newcommand{\hsnewpar}[1]%
  {{\parskip=0pt\parindent=0pt\par\vskip #1\noindent}}

\newcommand{\hscodestyle}{}

\newcommand{\sethscode}[1]%
  {\expandafter\let\expandafter\hscode\csname #1\endcsname
   \expandafter\let\expandafter\endhscode\csname end#1\endcsname}

  {\par\noindent
   \advance\leftskip\mathindent
   \hscodestyle
   \let\\=\@normalcr
   \let\hspre\(\let\hspost\)%
   \pboxed}%
  {\endpboxed\)%
   \par\noindent
   \ignorespacesafterend}

  {\hsnewpar\abovedisplayskip
   \advance\leftskip\mathindent
   \hscodestyle
   \let\hspre\(\let\hspost\)%
   \pboxed}%
  {\endpboxed%
   \hsnewpar\belowdisplayskip
   \ignorespacesafterend}

  {\hsnewpar\abovedisplayskip
   \advance\leftskip\mathindent
   \hscodestyle
   \let\\=\@normalcr
   \(\pboxed}%
  {\endpboxed\)%
   \hsnewpar\belowdisplayskip
   \ignorespacesafterend}

\newcommand{\plainhs}{\sethscode{plainhscode}}

\plainhs

  {\hsnewpar\abovedisplayskip
   \advance\leftskip\mathindent
   \hscodestyle
   \let\\=\@normalcr
   \(\parray}%
  {\endparray\)%
   \hsnewpar\belowdisplayskip
   \ignorespacesafterend}

  {\parray}{\endparray}

  {\(\parray}{\endparray\)}

\def\codeframewidth{\arrayrulewidth}
\RequirePackage{calc}

  {\parskip=\abovedisplayskip\par\noindent
   \hscodestyle
   \arrayrulewidth=\codeframewidth
   \tabular{@{}|p{\linewidth-2\arraycolsep-2\arrayrulewidth-2pt}|@{}}%
   \hline\framedhslinecorrect\\{-1.5ex}%
   \let\endoflinesave=\\
   \let\\=\@normalcr
   \(\pboxed}%
  {\endpboxed\)%
   \framedhslinecorrect\endoflinesave{.5ex}\hline
   \endtabular
   \parskip=\belowdisplayskip\par\noindent
   \ignorespacesafterend}

\newcommand{\framedhslinecorrect}[2]%
  {#1[#2]}

  {\(\def\column##1##2{}%
   \let\>\undefined\let\<\undefined\let\\\undefined
   \newcommand\>[1][]{}\newcommand\<[1][]{}\newcommand\\[1][]{}%
   \def\fromto##1##2##3{##3}%
   }{\) }%

  {\let\orighscode=\hscode
   \let\origendhscode=\endhscode
   \def\endhscode{\def\hscode{\endgroup\def\@currenvir{hscode}\\}\begingroup}
   \orighscode\def\hscode{\endgroup\def\@currenvir{hscode}}}%
  {\origendhscode
   \global\let\hscode=\orighscode
   \global\let\endhscode=\origendhscode}%

\makeatother
\EndFmtInput
\input{preamble}
\long\def\hide#1{}

\usepackage{newpxtext}
\usepackage{newpxmath}

\graphicspath{{figures/}}

\journal{arXiv}

\newcommand{\titel}{%
    Directories:
    \texorpdfstring{\\}{}
    A Convenient and Well-Behaved Formalism for Hierarchical Organization
    \texorpdfstring{\\}{}
    in Categorical Systems Theory
}

\hypersetup{%
	pdfauthor={%
    Owen Lynch
    and
		Markus Lohmayer
		},
  pdftitle=\titel{}
}

\begin{document}

\begin{frontmatter}
	\title{\titel}

  \author[topos]{Owen Lynch\corref{cor1}}
	\ead{owen@topos.institute}
  \cortext[cor1]{Corresponding author}

	\author[fau]{Markus Lohmayer}

	\address[topos]{%
		Topos Institute\\
    Berkeley, CA, USA
	}

	\address[fau]{%
		Institute of Applied Dynamics\\
   	Friedrich-Alexander Universität Erlangen-Nürnberg,
		Erlangen, Germany
	}

	\begin{abstract}
		\input{abstract.tex}
	\end{abstract}

\end{frontmatter}

\hide{
\begin{hscode}\SaveRestoreHook
\column{B}{@{}>{\hspre}l<{\hspost}@{}}%
\column{E}{@{}>{\hspre}l<{\hspost}@{}}%
\>[B]{}\mathbf{module}\;\Conid{Introduction}\;\mathbf{where}{}\<[E]%
\ColumnHook
\end{hscode}\resethooks
}

\section{Introduction}

\subsection{Motivation}

In the course of writing about exergetic port-Hamiltonian systems \cite{lohmayer_Exergetic_2025}, the second author names variables by period-separated identifiers, like $\sym{rotor.coil.flux}$. The first author, upon seeing this, thought that it was such a good idea that it should be formalized. Directories are the result of following that formalization to its logical conclusion from the perspective of category theory. We develop this formalization alongside an implementation in Haskell, in the hope that this will make the paper accessible both to programmers and mathematicians.

Is one to understand from the length of this paper that category theory vastly overcomplicates simple things? We hope that this is not the case. Rather, we have taken the opportunity to draw connections between the simple topic (directories) and a wide variety of topics in category theory such as polynomial monads, 2-algebra, the fam construction, etc., in the hope that directories provide a good example to demonstrate their practical application.

\subsection{Outline}

In the first part of the paper we develop a monad, $\Dtry$, on the category of sets. The intuition is that $\Dtry(X)$ is an assignment of some collection of names like $\sym{oscillator.mass.momentum}$ to elements of $X$.

In the second part, we then extend $\Dtry$ to be a 2-monad on the 2-category of categories. This is to provide an alternative, equivalent definition for symmetric monoidal categories and multicategories which allows for a more human-friendly monoidal product, where the factors of the monoidal product are labeled by names like $\sym{oscillator.spring.displacement}$. This can be achieved without sacrificing strictness (in the sense of strict algebras of a 2-monad).

We then include three appendices, which contain the statements and proofs of lemmas used in the main text.

\hide{
\begin{hscode}\SaveRestoreHook
\column{B}{@{}>{\hspre}l<{\hspost}@{}}%
\column{E}{@{}>{\hspre}l<{\hspost}@{}}%
\>[B]{}\mbox{\enskip\{-\# LANGUAGE GeneralizedNewtypeDeriving  \#-\}\enskip}{}\<[E]%
\\
\>[B]{}\mbox{\enskip\{-\# LANGUAGE DeriveFunctor  \#-\}\enskip}{}\<[E]%
\\
\>[B]{}\mbox{\enskip\{-\# LANGUAGE FlexibleInstances  \#-\}\enskip}{}\<[E]%
\\
\>[B]{}\mathbf{module}\;\Conid{\Conid{EPHS}.Directories}\;\mathbf{where}{}\<[E]%
\\[\blanklineskip]%
\>[B]{}\mathbf{import}\;\Conid{\Conid{Control}.Monad}\;(\Varid{join}){}\<[E]%
\\[\blanklineskip]%
\>[B]{}\mathbf{import}\;\Conid{\Conid{Data}.Map}\;(\Conid{Map}){}\<[E]%
\\
\>[B]{}\mathbf{import}\;\Varid{qualified}\;\Conid{\Conid{Data}.Map}\;\Varid{as}\;\Conid{Map}{}\<[E]%
\\
\>[B]{}\mathbf{import}\;\Conid{\Conid{Data}.List}\;(\Varid{sortOn},\Varid{foldl'}){}\<[E]%
\\[\blanklineskip]%
\>[B]{}\mathbf{import}\;\Conid{\Conid{Data}.Text}\;(\Conid{Text}){}\<[E]%
\\
\>[B]{}\mathbf{import}\;\Conid{\Conid{EPHS}.NonEmptyRecord}\;(\Conid{NonEmptyRecord}){}\<[E]%
\\
\>[B]{}\mathbf{import}\;\Varid{qualified}\;\Conid{\Conid{EPHS}.NonEmptyRecord}\;\Varid{as}\;\Conid{NonEmptyRecord}{}\<[E]%
\ColumnHook
\end{hscode}\resethooks
}

\section{Directories}

\subsection{Directories are maps from paths to values}

In short, a \textbf{directory} is an object in which \textbf{complete paths} are associated with \textbf{contents}. For instance, in the following directory the content of the complete path $\sym{oscillator.mass.p}$ is $\sym{0.0}$.
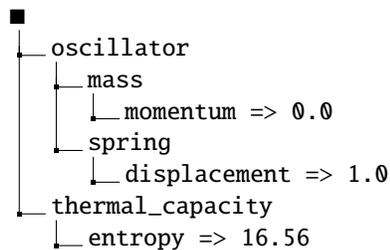
\dirtree{%
.1 $\blacksquare$.
.2 oscillator.
.3 mass.
.4 momentum => 0.0.
.3 spring.
.4 displacement => 1.0.
.2 thermal\_capacity.
.3 entropy => 16.56.
}

The underlying metaphor of directories should be familiar to anyone who has used a file system before. In this section, we will develop a mathematical formalism for directories.

Throughout this section, we will illustrate the mathematics using Haskell. Ideally, this is so that a reader with a background in functional programming but not math, or a reader with a background in math but not functional programming could understand the concepts. And the code is in fact real code, not pseudocode: this document is written with literate Haskell and thus can be used as a library.

From a mathematical standpoint, we build our definition of directories on top of the machinery of \emph{polynomial functors}, so we begin with a brief review of those. Using polynomial structures to describe data structures has a long history, going back at least to \cite{jay_Shapely_1994}. However, our notation and terminology is most influenced by the more recent treatment \cite{niu_Polynomial_2023}.

\begin{definition}
  Suppose that $A$ is a set. Let $y^A \colon \Set \to \Set$ be the functor defined by
  \begin{align*}
    y^A(X) &:= \{ x \colon A \to X \} = X^A \\
    y^A(f \colon X \to Y) &:= (x \colon X^A) \mapsto (x \cmp f \colon Y^A)
  \end{align*}
  where
  \[ A \xrightarrow{x \cmp f} Y = A \xrightarrow{x} X \xrightarrow{f} Y \]
\end{definition}

Some use the notation $y^A = \Hom_\Set(A, -)$ to define $y^A$; the above definition is simply an unfolding of what that means.\footnote{It turns out that the map $A \mapsto y^A$ is also a functor; it is a functor $\Set^\op \to \Set^\Set$ called the Yoneda embedding, which is foundational to many parts of category theory.}

\begin{definition}
  A \textbf{univariate polynomial functor} is a functor $p \colon \Set \to \Set$ such that there exists an indexed family of sets $\{F_i\}_{i \colon p(1)}$ with
  \[ p \cong \sum_{i \colon p(1)} y^{F_i} \]
  We call elements of $p(1)$ \textbf{positions} and elements of $F_i$ for $i \colon p(1)$ \textbf{directions at $i$}. The category $\Poly$ is the full subcategory of the functor category $\Set^\Set$ on the polynomial functors. That is, a morphism $p \to q$ in $\Poly$ is a natural transformation from $p$ to $q$.
\end{definition}

Henceforth, we will drop the ``univariate'' prefix and just call these polynomial functors, in line with the convention of Spivak and Niu. Almost everything we do with polynomial functors would not make sense if we did not have the following lemma.

\begin{lemma}
  If for a functor $p \colon \Set \to \Set$ such an indexed family $\{F_i\}_{i \colon p(1)}$ exists, then that indexed family is unique up to isomorphism. That is, if
  \[ \alpha \colon \sum_{i \colon p(1)} y^{F_i} \cong \sum_{i \colon p(1)} y^{F_i'} \]
  such that
  \[ \alpha_1 \colon \sum_{i \colon p(1)} 1^{F_i} \cong \sum_{i \colon p(1)} 1^{F_i'} \]
  is the identity on $\sum_{i \colon p(1)} 1$, then $F_i \cong F_i'$ for all $i \colon p(1)$.
\end{lemma}

We will not prove this, but the intuition for this is similar to how if you have a function $\bR \to \bR$ that can be written as a polynomial, it can only be written as a polynomial in one way.

For a polynomial functor $p$, we refer to the sets $F_i$ via the notation $p[i]$. Note that $p[i]$ is only defined up to isomorphism. The notation $p[i]$ emphasizes that this is not the application of $p$ to $i$; note that such an application would not typecheck as $i$ is not a set.

\begin{example}
  Basic polynomial expressions from highschool can be recast as polynomial functors in this framework. For instance, $y + 1$, $y^2 + 2y + 1$, and $3y^4 + 4y^3$ are all polynomial functors, where we interpret $y^n$ as $y^{\{1,\ldots,n\}}$, $1$ as $y^0$, and $m y^n$ as $\sum_{i \colon \{1,\ldots,m\}} y^{\{1,\ldots,n\}}$.
\end{example}

In \cite{jay_Shapely_1994}, polynomial functors are described as ``shapely functors.'' The idea is that a polynomial functor describes a generic data structure that can be split up into a ``shape'' and ``data.'' Given a polynomial $p$, a set $X$, and an element $(i \colon p(1), x \colon X^{p[i]}) \colon p(X)$, the ``shape'' is $i \colon p(1)$, and the ``data'' is $x \colon X^{p[i]}$.

\begin{example} \label{ex:list_polynomial}
  The generic data structure of lists can be described by the polynomial functor
  \[ \List := \sum_{n \colon \bN} y^n \]
  Given an element $l = (n \colon \bN, f \colon X^n) \colon \List(X)$, its shape is $n$ and its data is $f$.
\end{example}

A ``shapely type'' is a shapely functor $p$ where $p[i]$ is finite for all $i \colon p(1)$; $\List$ is a shapely type, and so is $y^2 + 2y + 1$.

\begin{example}
  Let $\Record$ be the polynomial functor defined by
  \[ \Record := \sum_{U \Subset \SymbolSet} y^U \]
  where $\Subset$ means ``finite subset''.

  Records are like lists, but indexed by names rather than numbers. A position of $\Record$ is a finite subset of some fixed set of symbols (which could be the set of ASCII strings, or the set of Unicode strings, or the set of Unicode strings not containing punctuation or spaces, etc.; in general all we assume is that the cardinality of $\SymbolSet$ is at least 2) and a direction at the position $U$ is an element of $U$. We capture this in the following Haskell code.
\begin{hscode}\SaveRestoreHook
\column{B}{@{}>{\hspre}l<{\hspost}@{}}%
\column{E}{@{}>{\hspre}l<{\hspost}@{}}%
\>[B]{}\mathbf{type}\;\Conid{Record}\;\Varid{a}\mathrel{=}\Conid{Map}\;\Conid{Text}\;\Varid{a}{}\<[E]%
\ColumnHook
\end{hscode}\resethooks
  Note that we don't use a sum to define this in Haskell, because Haskell has no dependent sum. But this type is still isomorphic to a hypothetical definition that used dependent sum, and that is what counts to be a polynomial functor.
\end{example}

We could use records to store directories, by using strings like \ensuremath{\text{\ttfamily \char34 a.b\char34}} as keys. However, certain operations (like filtering for paths whose first segment was \ensuremath{\text{\ttfamily \char34 a\char34}}) would involve error-prone string manipulation. An improvement would be something like

\begin{hscode}\SaveRestoreHook
\column{B}{@{}>{\hspre}l<{\hspost}@{}}%
\column{E}{@{}>{\hspre}l<{\hspost}@{}}%
\>[B]{}\mathbf{type}\;\Conid{Path}\mathrel{=}[\mskip1.5mu \Conid{Text}\mskip1.5mu]{}\<[E]%
\\[\blanklineskip]%
\>[B]{}\mathbf{type}\;\Dtry_0\;\Varid{a}\mathrel{=}\Conid{Map}\;\Conid{Path}\;\Varid{a}{}\<[E]%
\ColumnHook
\end{hscode}\resethooks

This has the advantage that we don't need to do string manipulation. However, there is a subtle problem with this definition of directory. It is best to understand this problem from the perspective of how we wish to use directories in systems theory: to index the variables in a system. That is, we want something like
\begin{hscode}\SaveRestoreHook
\column{B}{@{}>{\hspre}l<{\hspost}@{}}%
\column{3}{@{}>{\hspre}c<{\hspost}@{}}%
\column{3E}{@{}l@{}}%
\column{6}{@{}>{\hspre}l<{\hspost}@{}}%
\column{17}{@{}>{\hspre}c<{\hspost}@{}}%
\column{17E}{@{}l@{}}%
\column{21}{@{}>{\hspre}l<{\hspost}@{}}%
\column{E}{@{}>{\hspre}l<{\hspost}@{}}%
\>[B]{}\mathbf{data}\;\Conid{System}\mathrel{=}{}\<[E]%
\\
\>[B]{}\hsindent{3}{}\<[3]%
\>[3]{}\{\mskip1.5mu {}\<[3E]%
\>[6]{}\Varid{variables}{}\<[17]%
\>[17]{}\mathbin{::}{}\<[17E]%
\>[21]{}\Conid{Dtry}\;\Conid{QuantityType}{}\<[E]%
\\
\>[B]{}\hsindent{3}{}\<[3]%
\>[3]{},{}\<[3E]%
\>[6]{}\Varid{equations}{}\<[17]%
\>[17]{}\mathbin{::}{}\<[17E]%
\>[21]{}\mathbin{...}{}\<[E]%
\\
\>[B]{}\hsindent{3}{}\<[3]%
\>[3]{}\mskip1.5mu\}{}\<[3E]%
\ColumnHook
\end{hscode}\resethooks
However, we \emph{also} want to use directories to index the \emph{systems} in a composition specification. Composition might look like
\begin{hscode}\SaveRestoreHook
\column{B}{@{}>{\hspre}l<{\hspost}@{}}%
\column{3}{@{}>{\hspre}c<{\hspost}@{}}%
\column{3E}{@{}l@{}}%
\column{6}{@{}>{\hspre}l<{\hspost}@{}}%
\column{19}{@{}>{\hspre}c<{\hspost}@{}}%
\column{19E}{@{}l@{}}%
\column{23}{@{}>{\hspre}l<{\hspost}@{}}%
\column{E}{@{}>{\hspre}l<{\hspost}@{}}%
\>[B]{}\mathbf{data}\;\Conid{CompositionSpec}\mathrel{=}{}\<[E]%
\\
\>[B]{}\hsindent{3}{}\<[3]%
\>[3]{}\{\mskip1.5mu {}\<[3E]%
\>[6]{}\Varid{systems}{}\<[19]%
\>[19]{}\mathbin{::}{}\<[19E]%
\>[23]{}\Conid{Dtry}\;\Conid{System}{}\<[E]%
\\
\>[B]{}\hsindent{3}{}\<[3]%
\>[3]{},{}\<[3E]%
\>[6]{}\Varid{connections}{}\<[19]%
\>[19]{}\mathbin{::}{}\<[19E]%
\>[23]{}\mathbin{...}{}\<[E]%
\\
\>[B]{}\hsindent{3}{}\<[3]%
\>[3]{}\mskip1.5mu\}{}\<[3E]%
\\[\blanklineskip]%
\>[B]{}\Varid{compose}\mathbin{::}\Conid{CompositionSpec}\to \Conid{System}{}\<[E]%
\ColumnHook
\end{hscode}\resethooks
So in order to produce the directory of variables in the composed system, we need a function
\begin{hscode}\SaveRestoreHook
\column{B}{@{}>{\hspre}l<{\hspost}@{}}%
\column{E}{@{}>{\hspre}l<{\hspost}@{}}%
\>[B]{}\Varid{flatten}\mathbin{::}\Conid{Dtry}\;(\Conid{Dtry}\;\Conid{QuantityType})\to \Conid{Dtry}\;\Conid{QuantityType}{}\<[E]%
\ColumnHook
\end{hscode}\resethooks
The problem with \ensuremath{\Dtry_0} as a proposed definition for \ensuremath{\Conid{Dtry}} is that if system $\sym{a}$ has a variable $\sym{b.c}$ and system $\sym{a.b}$ has a variable $\sym{c}$, then \ensuremath{\Varid{flatten}} will confuse them! We need to systematically disallow this from happening; the following section will explain how to do this.

\subsection{Directories as tries}

We can solve the earlier problem with flattening via the following data structure. If the reader is unfamiliar with Haskell, and prefers the polynomial functor definition, then they should skip ahead to \cref{def:polynomial_monad}.
\begin{hscode}\SaveRestoreHook
\column{B}{@{}>{\hspre}l<{\hspost}@{}}%
\column{15}{@{}>{\hspre}c<{\hspost}@{}}%
\column{15E}{@{}l@{}}%
\column{18}{@{}>{\hspre}l<{\hspost}@{}}%
\column{E}{@{}>{\hspre}l<{\hspost}@{}}%
\>[B]{}\mathbf{data}\;\Dtry_1\;\Varid{a}{}\<[15]%
\>[15]{}\mathrel{=}{}\<[15E]%
\>[18]{}\Conid{Leaf}\;\Varid{a}{}\<[E]%
\\
\>[15]{}\mid {}\<[15E]%
\>[18]{}\Conid{Node}\;(\Conid{Record}\;(\Dtry_1\;\Varid{a})){}\<[E]%
\ColumnHook
\end{hscode}\resethooks

\noindent To give a feel for how we use \ensuremath{\Dtry_1}, we can look up a path inside a \ensuremath{\Dtry_1} via the following function.

\begin{hscode}\SaveRestoreHook
\column{B}{@{}>{\hspre}l<{\hspost}@{}}%
\column{3}{@{}>{\hspre}l<{\hspost}@{}}%
\column{12}{@{}>{\hspre}c<{\hspost}@{}}%
\column{12E}{@{}l@{}}%
\column{13}{@{}>{\hspre}l<{\hspost}@{}}%
\column{16}{@{}>{\hspre}l<{\hspost}@{}}%
\column{21}{@{}>{\hspre}l<{\hspost}@{}}%
\column{37}{@{}>{\hspre}l<{\hspost}@{}}%
\column{E}{@{}>{\hspre}l<{\hspost}@{}}%
\>[B]{}\Varid{lookupPath}\mathbin{::}\Conid{Path}\to \Dtry_1\;\Varid{a}\to \Conid{Maybe}\;(\Dtry_1\;\Varid{a}){}\<[E]%
\\
\>[B]{}\Varid{lookupPath}\;{}\<[13]%
\>[13]{}[\mskip1.5mu \mskip1.5mu]\;{}\<[21]%
\>[21]{}\Varid{d}{}\<[37]%
\>[37]{}\mathrel{=}\Conid{Just}\;\Varid{d}{}\<[E]%
\\
\>[B]{}\Varid{lookupPath}\;{}\<[13]%
\>[13]{}(\Varid{n}\mathbin{:}\Varid{ns})\;{}\<[21]%
\>[21]{}(\Conid{Node}\;\Varid{entries}){}\<[37]%
\>[37]{}\mathrel{=}\mathbf{case}\;\Varid{\Conid{Map}.lookup}\;\Varid{n}\;\Varid{entries}\;\mathbf{of}{}\<[E]%
\\
\>[B]{}\hsindent{3}{}\<[3]%
\>[3]{}\Conid{Just}\;\Varid{d}{}\<[12]%
\>[12]{}\to {}\<[12E]%
\>[16]{}\Varid{lookupPath}\;\Varid{ns}\;\Varid{d}{}\<[E]%
\\
\>[B]{}\hsindent{3}{}\<[3]%
\>[3]{}\Conid{Nothing}{}\<[12]%
\>[12]{}\to {}\<[12E]%
\>[16]{}\Conid{Nothing}{}\<[E]%
\\
\>[B]{}\Varid{lookupPath}\;{}\<[13]%
\>[13]{}\anonymous \;{}\<[21]%
\>[21]{}\anonymous {}\<[37]%
\>[37]{}\mathrel{=}\Conid{Nothing}{}\<[E]%
\ColumnHook
\end{hscode}\resethooks

\noindent Additionally, we can implement the \ensuremath{\Conid{Functor}} typeclass, building off of the \ensuremath{\Conid{Functor}} instance for \ensuremath{\Conid{Map}\;\Conid{Text}}.

\begin{hscode}\SaveRestoreHook
\column{B}{@{}>{\hspre}l<{\hspost}@{}}%
\column{3}{@{}>{\hspre}l<{\hspost}@{}}%
\column{E}{@{}>{\hspre}l<{\hspost}@{}}%
\>[B]{}\mathbf{instance}\;\Conid{Functor}\;\Dtry_1\;\mathbf{where}{}\<[E]%
\\
\>[B]{}\hsindent{3}{}\<[3]%
\>[3]{}\Varid{fmap}\;\Varid{f}\;(\Conid{Leaf}\;\Varid{x})\mathrel{=}\Conid{Leaf}\;(\Varid{f}\;\Varid{x}){}\<[E]%
\\
\>[B]{}\hsindent{3}{}\<[3]%
\>[3]{}\Varid{fmap}\;\Varid{f}\;(\Conid{Node}\;\Varid{r})\mathrel{=}\Conid{Node}\;(\Varid{fmap}\;(\Varid{fmap}\;\Varid{f})\;\Varid{r}){}\<[E]%
\ColumnHook
\end{hscode}\resethooks

The \ensuremath{\Varid{flatten}} function that we desire is one half of a monad structure on \ensuremath{\Dtry_1}. This paper will not add itself to the ranks of monad tutorials for Haskell, but we will give a brief reference for how the notation for monads in Haskell and monads in category theory lines up.

\begin{center}
  \begin{tabular}{ l | l | l }
    Monoid operation & Haskell & Category theory \\
    \hline
    Unit & \ensuremath{\Varid{return},\Varid{pure}\mathbin{::}\Varid{a}\to \Varid{m}\;\Varid{a}} & $\eta_X \colon X \to M(X)$ \\
    Multiplication & \ensuremath{\Varid{join},\Varid{flatten}\mathbin{::}\Varid{m}\;(\Varid{m}\;\Varid{a})\to \Varid{m}\;\Varid{a}} & $\mu_X \colon M(M(X)) \to M(X)$ \\
    N/A & \ensuremath{(\bind ),\Varid{bind}\mathbin{::}\Varid{m}\;\Varid{a}\to (\Varid{a}\to \Varid{m}\;\Varid{b})\to \Varid{m}\;\Varid{b}} & N/A
  \end{tabular}
\end{center}

In order to more closely align this paper with category theory notation, we will use a non-standard typeclass for monads in Haskell.\footnote{Using \ensuremath{\mathit{Monad}_\mu^\eta} has the side benefit that we need not mention applicatives beyond this footnote.}

\begin{hscode}\SaveRestoreHook
\column{B}{@{}>{\hspre}l<{\hspost}@{}}%
\column{3}{@{}>{\hspre}l<{\hspost}@{}}%
\column{8}{@{}>{\hspre}l<{\hspost}@{}}%
\column{20}{@{}>{\hspre}c<{\hspost}@{}}%
\column{20E}{@{}l@{}}%
\column{24}{@{}>{\hspre}l<{\hspost}@{}}%
\column{E}{@{}>{\hspre}l<{\hspost}@{}}%
\>[B]{}\mathbf{class}\;(\Conid{Functor}\;\Varid{m})\Rightarrow \mathit{Monad}_\mu^\eta\;\Varid{m}\;\mathbf{where}{}\<[E]%
\\
\>[B]{}\hsindent{3}{}\<[3]%
\>[3]{}\eta{}\<[8]%
\>[8]{}\mathbin{::}\Varid{a}{}\<[20]%
\>[20]{}\to {}\<[20E]%
\>[24]{}\Varid{m}\;\Varid{a}{}\<[E]%
\\
\>[B]{}\hsindent{3}{}\<[3]%
\>[3]{}\mu{}\<[8]%
\>[8]{}\mathbin{::}\Varid{m}\;(\Varid{m}\;\Varid{a}){}\<[20]%
\>[20]{}\to {}\<[20E]%
\>[24]{}\Varid{m}\;\Varid{a}{}\<[E]%
\ColumnHook
\end{hscode}\resethooks

\noindent We can then define an instance of \ensuremath{\mathit{Monad}_\mu^\eta} for \ensuremath{\Dtry_1}.

\begin{hscode}\SaveRestoreHook
\column{B}{@{}>{\hspre}l<{\hspost}@{}}%
\column{3}{@{}>{\hspre}l<{\hspost}@{}}%
\column{16}{@{}>{\hspre}c<{\hspost}@{}}%
\column{16E}{@{}l@{}}%
\column{19}{@{}>{\hspre}l<{\hspost}@{}}%
\column{E}{@{}>{\hspre}l<{\hspost}@{}}%
\>[B]{}\mathbf{instance}\;\mathit{Monad}_\mu^\eta\;\Dtry_1\;\mathbf{where}{}\<[E]%
\\
\>[B]{}\hsindent{3}{}\<[3]%
\>[3]{}\eta\mathrel{=}\Conid{Leaf}{}\<[E]%
\\
\>[B]{}\hsindent{3}{}\<[3]%
\>[3]{}\mu\;(\Conid{Leaf}\;\Varid{x}){}\<[16]%
\>[16]{}\mathrel{=}{}\<[16E]%
\>[19]{}\Varid{x}{}\<[E]%
\\
\>[B]{}\hsindent{3}{}\<[3]%
\>[3]{}\mu\;(\Conid{Node}\;\Varid{r}){}\<[16]%
\>[16]{}\mathrel{=}{}\<[16E]%
\>[19]{}\Conid{Node}\;(\Varid{fmap}\;\mu\;\Varid{r}){}\<[E]%
\ColumnHook
\end{hscode}\resethooks
If we try to directly translate the definition of \ensuremath{\Dtry_1} and its monad instance into math, it would be strictly less helpful than the Haskell version. However, by noticing that \ensuremath{\Dtry_1} comes from a \emph{free monad} construction, we both make less work for ourselves when we want to prove things and we discover an elegant mathematical story.

\begin{definition} \label{def:polynomial_monad}
  A \textbf{polynomial monad} is a monad $(M, \eta, \mu)$ on $\Set$ whose underlying functor is polynomial. The category $\PolyMon$ is the full subcategory of the category of monads on $\Set$ corresponding to the polynomial monads.
\end{definition}

\begin{example}
  The polynomial functor $\List$ defined in \cref{ex:list_polynomial} has a monad structure where $\eta_X \colon X \to \List(X)$ is defined by
  \[\eta_X(x) = [x]\]
  and $\mu_X \colon \List(\List(X)) \to \List(X)$ is defined by
  \[ \mu_X([[x_{1,1},\ldots,x_{1,k_1}],\ldots,[x_{n,1},\ldots,x_{n,k_n}]]) = [x_{1,1},\ldots,x_{n,k_n}] \qedhere \]
\end{example}

\begin{proposition}
  Let $U$ be the forgetful functor $\PolyMon \to \Poly$ which sends a monad $(m, \eta, \mu)$ to $m$. Then $U$ has a left adjoint which we denote by $f \mapsto \freemonad_f$. We call $\freemonad_f$ the \textbf{free monad} on $p$.
\end{proposition}

\begin{proof}
  We will not give a full proof here (the reader can refer to \cite{kelly_Unified_1980,nlabauthors_transfinite_2024,libkind_Pattern_2024} for a complete treatment), but we will give a construction of $\freemonad_p$ in Haskell and directly in $\Poly$. The Haskell is slightly easier to parse at a first glance, because we are allowed to express recursive types directly in Haskell instead of having to resort to an explicit construction of a fixed point.
\begin{hscode}\SaveRestoreHook
\column{B}{@{}>{\hspre}l<{\hspost}@{}}%
\column{3}{@{}>{\hspre}l<{\hspost}@{}}%
\column{16}{@{}>{\hspre}c<{\hspost}@{}}%
\column{16E}{@{}l@{}}%
\column{17}{@{}>{\hspre}c<{\hspost}@{}}%
\column{17E}{@{}l@{}}%
\column{19}{@{}>{\hspre}l<{\hspost}@{}}%
\column{20}{@{}>{\hspre}l<{\hspost}@{}}%
\column{21}{@{}>{\hspre}c<{\hspost}@{}}%
\column{21E}{@{}l@{}}%
\column{24}{@{}>{\hspre}l<{\hspost}@{}}%
\column{25}{@{}>{\hspre}l<{\hspost}@{}}%
\column{E}{@{}>{\hspre}l<{\hspost}@{}}%
\>[B]{}\mathbf{data}\;\Conid{Free}\;\Varid{f}\;\Varid{a}{}\<[16]%
\>[16]{}\mathrel{=}{}\<[16E]%
\>[19]{}\Conid{Pure}\;{}\<[25]%
\>[25]{}\Varid{a}{}\<[E]%
\\
\>[16]{}\mid {}\<[16E]%
\>[19]{}\Conid{Join}\;{}\<[25]%
\>[25]{}(\Varid{f}\;(\Conid{Free}\;\Varid{f}\;\Varid{a})){}\<[E]%
\\[\blanklineskip]%
\>[B]{}\mathbf{instance}\;(\Conid{Functor}\;\Varid{f})\Rightarrow \Conid{Functor}\;(\Conid{Free}\;\Varid{f})\;\mathbf{where}{}\<[E]%
\\
\>[B]{}\hsindent{3}{}\<[3]%
\>[3]{}\Varid{fmap}\;\Varid{g}\;(\Conid{Pure}\;\Varid{x}){}\<[21]%
\>[21]{}\mathrel{=}{}\<[21E]%
\>[24]{}\Conid{Pure}\;(\Varid{g}\;\Varid{x}){}\<[E]%
\\
\>[B]{}\hsindent{3}{}\<[3]%
\>[3]{}\Varid{fmap}\;\Varid{g}\;(\Conid{Join}\;\Varid{xs}){}\<[21]%
\>[21]{}\mathrel{=}{}\<[21E]%
\>[24]{}\Conid{Join}\;(\Varid{fmap}\;(\Varid{fmap}\;\Varid{g})\;\Varid{xs}){}\<[E]%
\\[\blanklineskip]%
\>[B]{}\mathbf{instance}\;(\Conid{Functor}\;\Varid{f})\Rightarrow \mathit{Monad}_\mu^\eta\;(\Conid{Free}\;\Varid{f})\;\mathbf{where}{}\<[E]%
\\
\>[B]{}\hsindent{3}{}\<[3]%
\>[3]{}\eta\mathrel{=}\Conid{Pure}{}\<[E]%
\\
\>[B]{}\hsindent{3}{}\<[3]%
\>[3]{}\mu\;(\Conid{Pure}\;\Varid{x}){}\<[17]%
\>[17]{}\mathrel{=}{}\<[17E]%
\>[20]{}\Varid{x}{}\<[E]%
\\
\>[B]{}\hsindent{3}{}\<[3]%
\>[3]{}\mu\;(\Conid{Join}\;\Varid{xs}){}\<[17]%
\>[17]{}\mathrel{=}{}\<[17E]%
\>[20]{}\Conid{Join}\;(\Varid{fmap}\;\mu\;\Varid{xs}){}\<[E]%
\ColumnHook
\end{hscode}\resethooks
  Note that this is almost the same as the definition of \ensuremath{\Dtry_1} above, only \ensuremath{\Conid{Record}} has been replaced with \ensuremath{\Varid{f}}. Specifically, we have exchanged \ensuremath{\Conid{Leaf}} for \ensuremath{\Conid{Pure}}, \ensuremath{\Conid{Node}} for \ensuremath{\Conid{Join}}, and \ensuremath{\Conid{Record}} for \ensuremath{\Varid{f}}.

  The equivalent mathematical construction works in the following way, essentially building the above type via a formula for a fixed point which uses a colimit. Assume that $p$ is a polynomial functor with $p[i]$ finite for each $i \colon p(1)$.\footnote{The construction for $p[i]$ possibly infinite is slightly more complicated, but can also be done: see \cite{libkind_Pattern_2024}.} Given a polynomial functor $p$, consider the following functor $p_{(-)} \colon (\bN, \leq) \to \Poly$. On objects, define $p_{(n)}$ by induction as

  \begin{align*}
    p_{(0)} &= y \\
    p_{(n+1)} &= y + p \circ p_{(n)}
  \end{align*}

  For arrows, define $p_{(0 \to 1)} \colon p_{(0)} \to p_{(1)}$ via the left inclusion $y \to y + p \circ y$. Then, assuming that we have defined $p_{(n \to n+1)}$, define
  \[p_{(n+1 \to n+2)} \colon p_{(n+1)} \to (p_{(n+2)} = y + p \circ p_{(n+1)})\]
  by using the functoriality of $+$ and $\circ$, that is
  \[p_{(n+1 \to n+2)} = 1_y + 1_p \circ p_{(n \to n+1)} \]
  Finally, define $\freemonad_p$ as the colimit of $p_{(-)}$:
  \[ \freemonad_p := \colim_{\bN} p_{(-)} \]
\end{proof}

We can now see that $\Dtry_1 \cong \freemonad_\Record$. This is a perfectly fine mathematical object, and it rules out the earlier problem with \ensuremath{\Dtry_0}; namely if \text{\ttfamily a\char46{}b\char46{}c} points to a value, then \text{\ttfamily a\char46{}b} cannot point to a value. However, in the process of solving one problem with \ensuremath{\Dtry_0}, we introduced a new problem. \ensuremath{\Dtry_0\;\Varid{a}} had the desirable property that an element of \ensuremath{\Dtry_0\;\Varid{a}} is fully characterized as a map from complete paths to values. This is no longer the case for \ensuremath{\Dtry_1\;\Varid{a}}, because we can have tree structures like
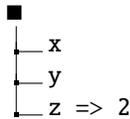
\dirtree{%
  .1 $\blacksquare$.
  .2 x.
  .2 y.
  .2 z => 2.
}
\noindent in which \ensuremath{\Varid{x}} and \ensuremath{\Varid{y}} are ``empty directories.'' This complicates defining operations like

\begin{hscode}\SaveRestoreHook
\column{B}{@{}>{\hspre}l<{\hspost}@{}}%
\column{E}{@{}>{\hspre}l<{\hspost}@{}}%
\>[B]{}\Varid{filter}\mathbin{::}(\Varid{a}\to \Conid{Bool})\to \Conid{Dtry}\;\Varid{a}\to \Conid{Dtry}\;\Varid{a}{}\<[E]%
\ColumnHook
\end{hscode}\resethooks

\noindent because it is ambiguous what to do when all contents of a directory are filtered out; do we leave it empty or delete it?

\subsection{Disallowing non-empty subdirectories}

The solution is to disallow empty directories, \emph{except at the top level}. We will see precisely what this means shortly. First we define a \ensuremath{\Conid{NonEmptyRecord}} type as follows.

\hide{
\begin{hscode}\SaveRestoreHook
\column{B}{@{}>{\hspre}l<{\hspost}@{}}%
\column{E}{@{}>{\hspre}l<{\hspost}@{}}%
\>[B]{}\mbox{\enskip\{-\# LANGUAGE GeneralizedNewtypeDeriving  \#-\}\enskip}{}\<[E]%
\\
\>[B]{}\mbox{\enskip\{-\# LANGUAGE DeriveFunctor  \#-\}\enskip}{}\<[E]%
\\
\>[B]{}\mbox{\enskip\{-\# LANGUAGE FlexibleInstances  \#-\}\enskip}{}\<[E]%
\ColumnHook
\end{hscode}\resethooks
}

\begin{hscode}\SaveRestoreHook
\column{B}{@{}>{\hspre}l<{\hspost}@{}}%
\column{E}{@{}>{\hspre}l<{\hspost}@{}}%
\>[B]{}\mathbf{module}\;\mathit{Record}_{\neq \emptyset}\;\mathbf{where}{}\<[E]%
\ColumnHook
\end{hscode}\resethooks
\hide{
\begin{hscode}\SaveRestoreHook
\column{B}{@{}>{\hspre}l<{\hspost}@{}}%
\column{E}{@{}>{\hspre}l<{\hspost}@{}}%
\>[B]{}\mathbf{import}\;\Conid{\Conid{Data}.Map}\;(\Conid{Map}){}\<[E]%
\\
\>[B]{}\mathbf{import}\;\Varid{qualified}\;\Conid{\Conid{Data}.Map}\;\Varid{as}\;\Conid{Map}{}\<[E]%
\\
\>[B]{}\mathbf{import}\;\Conid{\Conid{Data}.Text}\;(\Conid{Text}){}\<[E]%
\\[\blanklineskip]%
\>[B]{}\mathbf{type}\;\Conid{Record}\mathrel{=}\Conid{Map}\;\Conid{Text}{}\<[E]%
\ColumnHook
\end{hscode}\resethooks
}
\begin{hscode}\SaveRestoreHook
\column{B}{@{}>{\hspre}l<{\hspost}@{}}%
\column{3}{@{}>{\hspre}l<{\hspost}@{}}%
\column{E}{@{}>{\hspre}l<{\hspost}@{}}%
\>[B]{}\mathbf{newtype}\;\mathit{Record}_{\neq \emptyset}\;\Varid{a}\mathrel{=}\mathit{MkRecord}_{\neq \emptyset}\;(\Conid{Map}\;\Conid{Text}\;\Varid{a}){}\<[E]%
\\
\>[B]{}\hsindent{3}{}\<[3]%
\>[3]{}\mathbf{deriving}\;\Conid{Functor}{}\<[E]%
\\[\blanklineskip]%
\>[B]{}\Varid{coerce}\mathbin{::}\Conid{Record}\;\Varid{a}\to \Conid{Maybe}\;(\mathit{Record}_{\neq \emptyset}\;\Varid{a}){}\<[E]%
\\
\>[B]{}\Varid{coerce}\;\Varid{r}\mathrel{=}\mathbf{if}\;\Varid{\Conid{Map}.null}\;\Varid{r}\;\mathbf{then}\;\Conid{Nothing}\;\mathbf{else}\;\Conid{Just}\;(\mathit{MkRecord}_{\neq \emptyset}\;\Varid{r}){}\<[E]%
\\[\blanklineskip]%
\>[B]{}\Varid{toRecord}\mathbin{::}\mathit{Record}_{\neq \emptyset}\;\Varid{a}\to \Conid{Record}\;\Varid{a}{}\<[E]%
\\
\>[B]{}\Varid{toRecord}\;(\mathit{MkRecord}_{\neq \emptyset}\;\Varid{r})\mathrel{=}\Varid{r}{}\<[E]%
\ColumnHook
\end{hscode}\resethooks
\hide{
\begin{hscode}\SaveRestoreHook
\column{B}{@{}>{\hspre}l<{\hspost}@{}}%
\column{22}{@{}>{\hspre}l<{\hspost}@{}}%
\column{E}{@{}>{\hspre}l<{\hspost}@{}}%
\>[B]{}\Varid{singleton}\mathbin{::}\Conid{Text}\to \Varid{a}\to \mathit{Record}_{\neq \emptyset}\;\Varid{a}{}\<[E]%
\\
\>[B]{}\Varid{singleton}\;\Varid{k}\;\Varid{x}\mathrel{=}\mathit{MkRecord}_{\neq \emptyset}\;(\Varid{\Conid{Map}.singleton}\;\Varid{k}\;\Varid{x}){}\<[E]%
\\[\blanklineskip]%
\>[B]{}\Varid{insert}\mathbin{::}\Conid{Text}\to \Varid{a}{}\<[22]%
\>[22]{}\to \mathit{Record}_{\neq \emptyset}\;\Varid{a}\to \mathit{Record}_{\neq \emptyset}\;\Varid{a}{}\<[E]%
\\
\>[B]{}\Varid{insert}\;\Varid{k}\;\Varid{x}\;(\mathit{MkRecord}_{\neq \emptyset}\;\Varid{m})\mathrel{=}\mathit{MkRecord}_{\neq \emptyset}\;(\Varid{\Conid{Map}.insert}\;\Varid{k}\;\Varid{x}\;\Varid{m}){}\<[E]%
\ColumnHook
\end{hscode}\resethooks
}

So long as we only construct a \ensuremath{\mathit{Record}_{\neq \emptyset}} via \ensuremath{\Varid{coerce}}, the invariant of being non-empty will remain true.

\noindent Mathematically, this corresponds to

\[ \NERecord := \sum_{\emptyset \neq U \Subset \SymbolSet} y^U \]

\noindent We use this \ensuremath{\mathit{Record}_{\neq \emptyset}} type to construct the definition for \ensuremath{\Conid{Dtry}}.

\begin{definition}
Let $\Dtry$ be the polynomial functor defined by

\[ \Dtry := \Maybe \circ \freemonad_{\NERecord} \]

\noindent where $\Maybe = y + 1$.
\end{definition}

\noindent This corresponds to the code

\begin{hscode}\SaveRestoreHook
\column{B}{@{}>{\hspre}l<{\hspost}@{}}%
\column{3}{@{}>{\hspre}l<{\hspost}@{}}%
\column{E}{@{}>{\hspre}l<{\hspost}@{}}%
\>[B]{}\mathbf{type}\;\Dtry_{\neq \emptyset}\;\Varid{a}\mathrel{=}\Conid{Free}\;\mathit{Record}_{\neq \emptyset}\;\Varid{a}{}\<[E]%
\\[\blanklineskip]%
\>[B]{}\mathbf{newtype}\;\Conid{Dtry}\;\Varid{a}\mathrel{=}\Conid{MkDtry}\;(\Conid{Maybe}\;(\Dtry_{\neq \emptyset}\;\Varid{a})){}\<[E]%
\\
\>[B]{}\hsindent{3}{}\<[3]%
\>[3]{}\mathbf{deriving}\;\Conid{Functor}{}\<[E]%
\ColumnHook
\end{hscode}\resethooks

\begin{proposition}
$\Dtry$ is a polynomial monad.
\end{proposition}

\begin{proof}
Although both $\Maybe$ and $\freemonad_{\NERecord}$ are monads, their composite is not automatically a monad. One way in order to put a monad structure on the composite is to use a \emph{distributive law}, which is a natural transformation of the following type \cite{beck_Distributive_1969}:

\[ \freemonad_{\NERecord} \circ \Maybe \to \Maybe \circ \freemonad_{\NERecord} \]

This takes a non-empty directory of \ensuremath{\Conid{Maybe}}s, and filters out all of the \ensuremath{\Conid{Nothing}} values. If there's anything left, it returns a non-empty directory of those; otherwise it returns \ensuremath{\Conid{Nothing}}. We can write this in Haskell:

\begin{hscode}\SaveRestoreHook
\column{B}{@{}>{\hspre}l<{\hspost}@{}}%
\column{20}{@{}>{\hspre}c<{\hspost}@{}}%
\column{20E}{@{}l@{}}%
\column{23}{@{}>{\hspre}l<{\hspost}@{}}%
\column{E}{@{}>{\hspre}l<{\hspost}@{}}%
\>[B]{}\Varid{filterNothings}\mathbin{::}\mathit{Record}_{\neq \emptyset}\;(\Conid{Maybe}\;\Varid{a})\to \Conid{Maybe}\;(\mathit{Record}_{\neq \emptyset}\;\Varid{a}){}\<[E]%
\\
\>[B]{}\Varid{filterNothings}\;\Varid{ds}\mathrel{=}\Varid{\mathit{Record}_{\neq \emptyset}.coerce}\;(\Varid{\Conid{Map}.mapMaybe}\;\Varid{id}\;(\Varid{\mathit{Record}_{\neq \emptyset}.toRecord}\;\Varid{ds})){}\<[E]%
\\[\blanklineskip]%
\>[B]{}\Varid{distrib}\mathbin{::}\Dtry_{\neq \emptyset}\;(\Conid{Maybe}\;\Varid{a})\to \Conid{Maybe}\;(\Dtry_{\neq \emptyset}\;\Varid{a}){}\<[E]%
\\
\>[B]{}\Varid{distrib}\;(\Conid{Pure}\;\Varid{mx}){}\<[20]%
\>[20]{}\mathrel{=}{}\<[20E]%
\>[23]{}\Varid{fmap}\;\Conid{Pure}\;\Varid{mx}{}\<[E]%
\\
\>[B]{}\Varid{distrib}\;(\Conid{Join}\;\Varid{ds}){}\<[20]%
\>[20]{}\mathrel{=}{}\<[20E]%
\>[23]{}\Varid{fmap}\;\Conid{Join}\;(\Varid{filterNothings}\;(\Varid{fmap}\;\Varid{distrib}\;\Varid{ds})){}\<[E]%
\ColumnHook
\end{hscode}\resethooks

\noindent where

\begin{hscode}\SaveRestoreHook
\column{B}{@{}>{\hspre}l<{\hspost}@{}}%
\column{E}{@{}>{\hspre}l<{\hspost}@{}}%
\>[B]{}\Varid{\Conid{Map}.mapMaybe}\mathbin{::}\Conid{Map}\;\Varid{a}\to (\Varid{a}\to \Conid{Maybe}\;\Varid{b})\to \Conid{Map}\;\Varid{b}{}\<[E]%
\ColumnHook
\end{hscode}\resethooks

\noindent We can then use \ensuremath{\Varid{distrib}} to build the Monad instance.

\hide{
\begin{hscode}\SaveRestoreHook
\column{B}{@{}>{\hspre}l<{\hspost}@{}}%
\column{3}{@{}>{\hspre}l<{\hspost}@{}}%
\column{E}{@{}>{\hspre}l<{\hspost}@{}}%
\>[B]{}\mathbf{instance}\;\mathit{Monad}_\mu^\eta\;\Conid{Maybe}\;\mathbf{where}{}\<[E]%
\\
\>[B]{}\hsindent{3}{}\<[3]%
\>[3]{}\eta\mathrel{=}\Varid{return}{}\<[E]%
\\
\>[B]{}\hsindent{3}{}\<[3]%
\>[3]{}\mu\mathrel{=}\Varid{join}{}\<[E]%
\ColumnHook
\end{hscode}\resethooks
}

\begin{hscode}\SaveRestoreHook
\column{B}{@{}>{\hspre}l<{\hspost}@{}}%
\column{3}{@{}>{\hspre}l<{\hspost}@{}}%
\column{5}{@{}>{\hspre}l<{\hspost}@{}}%
\column{7}{@{}>{\hspre}l<{\hspost}@{}}%
\column{E}{@{}>{\hspre}l<{\hspost}@{}}%
\>[B]{}\mathbf{instance}\;\mathit{Monad}_\mu^\eta\;\Conid{Dtry}\;\mathbf{where}{}\<[E]%
\\
\>[B]{}\hsindent{3}{}\<[3]%
\>[3]{}\eta\mathrel{=}\Conid{MkDtry}\mathbin{\circ}\Conid{Just}\mathbin{\circ}\Conid{Pure}{}\<[E]%
\\
\>[B]{}\hsindent{3}{}\<[3]%
\>[3]{}\mu\;\Varid{d}\mathrel{=}\Conid{MkDtry}\;(\Varid{fmap}\;\mu\;(\mu\;\Varid{swapped})){}\<[E]%
\\
\>[3]{}\hsindent{2}{}\<[5]%
\>[5]{}\mathbf{where}{}\<[E]%
\\
\>[5]{}\hsindent{2}{}\<[7]%
\>[7]{}(\Conid{MkDtry}\;\Varid{unwrapped})\mathrel{=}\Varid{fmap}\;(\lambda (\Conid{MkDtry}\;\Varid{d'})\to \Varid{d'})\;\Varid{d}{}\<[E]%
\\
\>[5]{}\hsindent{2}{}\<[7]%
\>[7]{}\Varid{swapped}\mathrel{=}\Varid{fmap}\;\Varid{distrib}\;\Varid{unwrapped}{}\<[E]%
\ColumnHook
\end{hscode}\resethooks

\noindent To understand the above code, consider the types of \ensuremath{\Varid{unwrapped}} and \ensuremath{\Varid{swapped}}:

\begin{hscode}\SaveRestoreHook
\column{B}{@{}>{\hspre}l<{\hspost}@{}}%
\column{12}{@{}>{\hspre}c<{\hspost}@{}}%
\column{12E}{@{}l@{}}%
\column{16}{@{}>{\hspre}l<{\hspost}@{}}%
\column{E}{@{}>{\hspre}l<{\hspost}@{}}%
\>[B]{}\Varid{unwrapped}{}\<[12]%
\>[12]{}\mathbin{::}{}\<[12E]%
\>[16]{}\Conid{Maybe}\;(\Dtry_{\neq \emptyset}\;(\Conid{Maybe}\;(\Dtry_{\neq \emptyset}\;\Varid{a}))){}\<[E]%
\\
\>[B]{}\Varid{swapped}{}\<[12]%
\>[12]{}\mathbin{::}{}\<[12E]%
\>[16]{}\Conid{Maybe}\;(\Conid{Maybe}\;(\Dtry_{\neq \emptyset}\;(\Dtry_{\neq \emptyset}\;\Varid{a}))){}\<[E]%
\ColumnHook
\end{hscode}\resethooks

\noindent The real work in defining $\mu$ is getting from \ensuremath{\Varid{unwrapped}} to \ensuremath{\Varid{swapped}} via the distributive law: once we have got to \ensuremath{\Varid{swapped}} we can use the monad structures of \ensuremath{\Conid{Maybe}} and \ensuremath{\Dtry_{\neq \emptyset}} to get the rest of the way. We show that \ensuremath{\Varid{distrib}} is actually a distributive law in \cref{thm:nerecord-maybe-distributive-law}.
\end{proof}

\subsection{Using $\Dtry$}

In this section, we go over some useful lemmas and associated functions for \ensuremath{\Conid{Dtry}} which make the task of reasoning with \ensuremath{\Conid{Dtry}} nicer. First of all, it is convenient to be able to try and construct a \ensuremath{\Conid{Dtry}\;\Varid{a}} from a  list of pairs of paths and values, i.e. something like \ensuremath{[\mskip1.5mu ([\mskip1.5mu \Conid{Text}\mskip1.5mu],\Varid{a})\mskip1.5mu]}, because when using \ensuremath{\Conid{Dtry}} in a domain specific language you often might want to parse things like

\begin{tabbing}\ttfamily
~~~a\char46{}x~\char61{}~2\\
\ttfamily ~~~a\char46{}y~\char61{}~1\\
\ttfamily ~~~b~~~\char61{}~3
\end{tabbing}

We do this mathematically by showing that there is an injection \ensuremath{\Varid{pathMap}\mathbin{::}\Conid{Dtry}\;\Varid{a}\to \Dtry_0\;\Varid{a}}, where we recall that \ensuremath{\Dtry_0\;\Varid{a}\mathrel{=}\Conid{Map}\;[\mskip1.5mu \Conid{Text}\mskip1.5mu]\;\Varid{a}}. We then implement the partial inverse \ensuremath{\Varid{fromPathMap}\mathbin{::}\Dtry_0\;\Varid{a}\to \Conid{Maybe}\;(\Conid{Dtry}\;\Varid{a})}.

We can define $\pMap$ with the following Haskell code.

\begin{hscode}\SaveRestoreHook
\column{B}{@{}>{\hspre}l<{\hspost}@{}}%
\column{3}{@{}>{\hspre}l<{\hspost}@{}}%
\column{5}{@{}>{\hspre}l<{\hspost}@{}}%
\column{7}{@{}>{\hspre}l<{\hspost}@{}}%
\column{25}{@{}>{\hspre}c<{\hspost}@{}}%
\column{25E}{@{}l@{}}%
\column{28}{@{}>{\hspre}l<{\hspost}@{}}%
\column{30}{@{}>{\hspre}c<{\hspost}@{}}%
\column{30E}{@{}l@{}}%
\column{33}{@{}>{\hspre}l<{\hspost}@{}}%
\column{E}{@{}>{\hspre}l<{\hspost}@{}}%
\>[B]{}\Varid{pathMap}\mathbin{::}\Conid{Dtry}\;\Varid{a}\to \Dtry_0\;\Varid{a}{}\<[E]%
\\
\>[B]{}\Varid{pathMap}\mathrel{=}\Varid{\Conid{Map}.fromList}\mathbin{\circ}\Varid{paths}{}\<[E]%
\\
\>[B]{}\hsindent{3}{}\<[3]%
\>[3]{}\mathbf{where}{}\<[E]%
\\
\>[3]{}\hsindent{2}{}\<[5]%
\>[5]{}\Varid{paths}\mathbin{::}\Conid{Dtry}\;\Varid{a}\to [\mskip1.5mu ([\mskip1.5mu \Conid{Text}\mskip1.5mu],\Varid{a})\mskip1.5mu]{}\<[E]%
\\
\>[3]{}\hsindent{2}{}\<[5]%
\>[5]{}\Varid{paths}\;(\Conid{MkDtry}\;\Conid{Nothing}){}\<[30]%
\>[30]{}\mathrel{=}{}\<[30E]%
\>[33]{}[\mskip1.5mu \mskip1.5mu]{}\<[E]%
\\
\>[3]{}\hsindent{2}{}\<[5]%
\>[5]{}\Varid{paths}\;(\Conid{MkDtry}\;(\Conid{Just}\;\Varid{d})){}\<[30]%
\>[30]{}\mathrel{=}{}\<[30E]%
\>[33]{}\Varid{paths'}\;\Varid{d}{}\<[E]%
\\
\>[3]{}\hsindent{2}{}\<[5]%
\>[5]{}\Varid{paths'}\mathbin{::}\Dtry_{\neq \emptyset}\;\Varid{a}\to [\mskip1.5mu ([\mskip1.5mu \Conid{Text}\mskip1.5mu],\Varid{a})\mskip1.5mu]{}\<[E]%
\\
\>[3]{}\hsindent{2}{}\<[5]%
\>[5]{}\Varid{paths'}\;(\Conid{Pure}\;\Varid{x}){}\<[25]%
\>[25]{}\mathrel{=}{}\<[25E]%
\>[28]{}[\mskip1.5mu ([\mskip1.5mu \mskip1.5mu],\Varid{x})\mskip1.5mu]{}\<[E]%
\\
\>[3]{}\hsindent{2}{}\<[5]%
\>[5]{}\Varid{paths'}\;((\Conid{Join}\;\Varid{ds})){}\<[25]%
\>[25]{}\mathrel{=}{}\<[25E]%
\>[28]{}\mathbf{do}{}\<[E]%
\\
\>[5]{}\hsindent{2}{}\<[7]%
\>[7]{}(\Varid{k},\Varid{d})\leftarrow \Varid{\Conid{Map}.toList}\mathbin{\$}\Varid{\mathit{Record}_{\neq \emptyset}.toRecord}\;\Varid{ds}{}\<[E]%
\\
\>[5]{}\hsindent{2}{}\<[7]%
\>[7]{}(\Varid{p},\Varid{x})\leftarrow \Varid{paths'}\;\Varid{d}{}\<[E]%
\\
\>[5]{}\hsindent{2}{}\<[7]%
\>[7]{}\Varid{return}\;(\Varid{k}\mathbin{:}\Varid{p},\Varid{x}){}\<[E]%
\ColumnHook
\end{hscode}\resethooks

\noindent Mathematically speaking, $\pMap$ is a natural transformation of type
\[ \pMap \colon \Dtry \into \sum_{U \Subset \SymbolSet^{\ast}} y^{U}\]
where $\SymbolSet^{\ast}$ is the set of lists of symbols. It can be defined by noticing that for $d \colon \Dtry(1)$, the set of directions $\Dtry[d]$ is naturally a subset of $\SymbolSet^{\ast}$. This is because a direction in $\Dtry[d]$ is a path through the tree $d$, which can be given as a list of symbols. Moreover, $\pMap$ is an injection because in a \ensuremath{\Conid{Dtry}}, the collection of internal nodes is determined purely by the paths to the leaves (which is not the case for \ensuremath{\Dtry_1}).

We will now construct a partial inverse to $\pMap$. This relies on first characterizing the image of $\pMap$.

\begin{definition}
  Given two paths $p, p' \colon \SymbolSet^{\ast}$, where $p = [p_{1},\ldots,p_{n}]$ and $p' = [p_{1}',\ldots,p_{n'}']$, $p$ is a \textbf{prefix} of $p'$ if $n \leq n'$ and $p_{i} = p'_{i}$ for $i \colon \{1,\ldots,n\}$. A subset $U \Subset \SymbolSet^{\ast}$ is \textbf{prefix-free} if for all pairs $p \neq p' \colon U$, $p$ is not a prefix of $p'$.
\end{definition}

\begin{proposition}
  The image of $\pMap$ consists of
  \[ \sum_{\substack{U \Subset \SymbolSet^{\ast} \\ \text{$U$ is prefix-free}}} y^{U} \]
\end{proposition}

\begin{proof}
  Let $d \colon \Dtry(1)$ be a directory. Then assume for a contradiction that we have $p,p' \colon \Dtry[d]$ with $p$ a prefix of $p'$. As $p'$ points to a leaf node in the tree for $d$, $p$ must point to an internal node in $d$. But then we can't have $p \colon \Dtry[d]$, and we are done.
\end{proof}

\begin{proposition}
  There is an inverse to
  \[ \pMap \colon \Dtry \to \sum_{\substack{U \Subset \SymbolSet^{\ast} \\ \text{$U$ is prefix-free}}} y^{U}\]
\end{proposition}

\begin{proof}
  We construct this inverse in Haskell via an accumulation of utility methods. Each utility method has a ``primed'' version, which is the version which operates on \ensuremath{\Dtry_{\neq \emptyset}} (which is the free monad on \ensuremath{\mathit{Record}_{\neq \emptyset}}), and an ``unprimed'' version which operates on \ensuremath{\Conid{Dtry}}.

  First, we have \ensuremath{\Varid{prefix}}, which takes in a directory and prefixes all of the paths in that directory by a single name.
\begin{hscode}\SaveRestoreHook
\column{B}{@{}>{\hspre}l<{\hspost}@{}}%
\column{9}{@{}>{\hspre}l<{\hspost}@{}}%
\column{12}{@{}>{\hspre}l<{\hspost}@{}}%
\column{25}{@{}>{\hspre}c<{\hspost}@{}}%
\column{25E}{@{}l@{}}%
\column{28}{@{}>{\hspre}l<{\hspost}@{}}%
\column{E}{@{}>{\hspre}l<{\hspost}@{}}%
\>[B]{}\Varid{prefix'}\mathbin{::}\Conid{Text}\to \Dtry_{\neq \emptyset}\;\Varid{a}\to \Dtry_{\neq \emptyset}\;\Varid{a}{}\<[E]%
\\
\>[B]{}\Varid{prefix'}\;\Varid{k}\;\Varid{d}\mathrel{=}\Conid{Join}\;(\Varid{\mathit{Record}_{\neq \emptyset}.singleton}\;\Varid{k}\;\Varid{d}){}\<[E]%
\\[\blanklineskip]%
\>[B]{}\Varid{prefix}\mathbin{::}\Conid{Text}\to \Conid{Dtry}\;\Varid{a}\to \Conid{Dtry}\;\Varid{a}{}\<[E]%
\\
\>[B]{}\Varid{prefix}\;{}\<[9]%
\>[9]{}\Varid{k}\;{}\<[12]%
\>[12]{}(\Conid{MkDtry}\;\Varid{m}){}\<[25]%
\>[25]{}\mathrel{=}{}\<[25E]%
\>[28]{}\Conid{MkDtry}\;(\Varid{fmap}\;(\Varid{prefix'}\;\Varid{k})\;\Varid{d}){}\<[E]%
\ColumnHook
\end{hscode}\resethooks

  Then we have \ensuremath{\Varid{singleton}}, which produces a directory with a single path in it leading to a certain key.
\begin{hscode}\SaveRestoreHook
\column{B}{@{}>{\hspre}l<{\hspost}@{}}%
\column{13}{@{}>{\hspre}l<{\hspost}@{}}%
\column{21}{@{}>{\hspre}l<{\hspost}@{}}%
\column{24}{@{}>{\hspre}l<{\hspost}@{}}%
\column{E}{@{}>{\hspre}l<{\hspost}@{}}%
\>[B]{}\Varid{singleton'}{}\<[13]%
\>[13]{}\mathbin{::}[\mskip1.5mu \Conid{Text}\mskip1.5mu]\to \Varid{a}\to \Dtry_{\neq \emptyset}\;\Varid{a}{}\<[E]%
\\
\>[B]{}\Varid{singleton'}\;{}\<[13]%
\>[13]{}[\mskip1.5mu \mskip1.5mu]\;{}\<[21]%
\>[21]{}\Varid{x}{}\<[24]%
\>[24]{}\mathrel{=}\eta\;\Varid{x}{}\<[E]%
\\
\>[B]{}\Varid{singleton'}\;{}\<[13]%
\>[13]{}(\Varid{k}\mathbin{:}\Varid{ks})\;{}\<[21]%
\>[21]{}\Varid{x}{}\<[24]%
\>[24]{}\mathrel{=}\Varid{prefix'}\;\Varid{k}\;(\Varid{singleton'}\;\Varid{ks}\;\Varid{x}){}\<[E]%
\\[\blanklineskip]%
\>[B]{}\Varid{singleton}\mathbin{::}[\mskip1.5mu \Conid{Text}\mskip1.5mu]\to \Varid{a}\to \Conid{Dtry}\;\Varid{a}{}\<[E]%
\\
\>[B]{}\Varid{singleton}\;\Varid{p}\;\Varid{x}\mathrel{=}\Conid{MkDtry}\;(\Conid{Just}\;(\Varid{singleton'}\;\Varid{p}\;\Varid{x})){}\<[E]%
\ColumnHook
\end{hscode}\resethooks

  Finally, we have \ensuremath{\Varid{insert}}, which takes a path, a value, and a directory and attempts to insert the value into the directory at the given path. It fails if the given path is a prefix of any pre-existing path in the directory, which includes the case that the path already exists in the directory.
\begin{hscode}\SaveRestoreHook
\column{B}{@{}>{\hspre}l<{\hspost}@{}}%
\column{3}{@{}>{\hspre}l<{\hspost}@{}}%
\column{5}{@{}>{\hspre}l<{\hspost}@{}}%
\column{10}{@{}>{\hspre}l<{\hspost}@{}}%
\column{12}{@{}>{\hspre}l<{\hspost}@{}}%
\column{13}{@{}>{\hspre}l<{\hspost}@{}}%
\column{18}{@{}>{\hspre}l<{\hspost}@{}}%
\column{21}{@{}>{\hspre}l<{\hspost}@{}}%
\column{31}{@{}>{\hspre}l<{\hspost}@{}}%
\column{32}{@{}>{\hspre}c<{\hspost}@{}}%
\column{32E}{@{}l@{}}%
\column{35}{@{}>{\hspre}l<{\hspost}@{}}%
\column{E}{@{}>{\hspre}l<{\hspost}@{}}%
\>[B]{}\Varid{insert'}\mathbin{::}[\mskip1.5mu \Conid{Text}\mskip1.5mu]\to \Varid{a}\to \Dtry_{\neq \emptyset}\;\Varid{a}\to \Conid{Maybe}\;(\Dtry_{\neq \emptyset}\;\Varid{a}){}\<[E]%
\\
\>[B]{}\Varid{insert'}\;{}\<[10]%
\>[10]{}\anonymous \;{}\<[18]%
\>[18]{}\anonymous \;{}\<[21]%
\>[21]{}(\Conid{Pure}\;\anonymous ){}\<[31]%
\>[31]{}\mathrel{=}\Conid{Nothing}{}\<[E]%
\\
\>[B]{}\Varid{insert'}\;{}\<[10]%
\>[10]{}[\mskip1.5mu \mskip1.5mu]\;{}\<[18]%
\>[18]{}\anonymous \;{}\<[21]%
\>[21]{}(\Conid{Join}\;\Varid{r}){}\<[31]%
\>[31]{}\mathrel{=}\Conid{Nothing}{}\<[E]%
\\
\>[B]{}\Varid{insert'}\;{}\<[10]%
\>[10]{}(\Varid{k}\mathbin{:}\Varid{ks})\;{}\<[18]%
\>[18]{}\Varid{x}\;{}\<[21]%
\>[21]{}(\Conid{Join}\;\Varid{r}){}\<[31]%
\>[31]{}\mathrel{=}\mathbf{case}\;\Varid{\Conid{Map}.lookup}\;\Varid{k}\;(\Varid{\mathit{Record}_{\neq \emptyset}.toRecord}\;\Varid{r})\;\mathbf{of}{}\<[E]%
\\
\>[B]{}\hsindent{3}{}\<[3]%
\>[3]{}\Conid{Just}\;\Varid{d}{}\<[12]%
\>[12]{}\to \mathbf{do}{}\<[E]%
\\
\>[3]{}\hsindent{2}{}\<[5]%
\>[5]{}\Varid{d'}\leftarrow \Varid{insert'}\;\Varid{ks}\;\Varid{x}\;\Varid{d}{}\<[E]%
\\
\>[3]{}\hsindent{2}{}\<[5]%
\>[5]{}\Varid{return}\mathbin{\$}\Conid{Join}\mathbin{\$}\Varid{\mathit{Record}_{\neq \emptyset}.insert}\;\Varid{k}\;\Varid{d'}\;\Varid{r}{}\<[E]%
\\
\>[B]{}\hsindent{3}{}\<[3]%
\>[3]{}\Conid{Nothing}{}\<[12]%
\>[12]{}\to \Conid{Just}\mathbin{\$}\Conid{Join}\mathbin{\$}\Varid{\mathit{Record}_{\neq \emptyset}.insert}\;\Varid{k}\;(\Varid{singleton'}\;\Varid{ks}\;\Varid{x})\;\Varid{r}{}\<[E]%
\\[\blanklineskip]%
\>[B]{}\Varid{insert}\mathbin{::}[\mskip1.5mu \Conid{Text}\mskip1.5mu]\to \Varid{a}\to \Conid{Dtry}\;\Varid{a}\to \Conid{Maybe}\;(\Conid{Dtry}\;\Varid{a}){}\<[E]%
\\
\>[B]{}\Varid{insert}\;\Varid{p}\;\Varid{x}\;{}\<[13]%
\>[13]{}(\Conid{MkDtry}\;\Conid{Nothing}){}\<[32]%
\>[32]{}\mathrel{=}{}\<[32E]%
\>[35]{}\Conid{Just}\;(\Varid{singleton}\;\Varid{p}\;\Varid{x}){}\<[E]%
\\
\>[B]{}\Varid{insert}\;\Varid{p}\;\Varid{x}\;{}\<[13]%
\>[13]{}(\Conid{MkDtry}\;(\Conid{Just}\;\Varid{d})){}\<[32]%
\>[32]{}\mathrel{=}{}\<[32E]%
\>[35]{}\mathbf{do}{}\<[E]%
\\
\>[B]{}\hsindent{3}{}\<[3]%
\>[3]{}\Varid{d'}\leftarrow \Varid{insert'}\;\Varid{p}\;\Varid{x}\;\Varid{d}{}\<[E]%
\\
\>[B]{}\hsindent{3}{}\<[3]%
\>[3]{}\Varid{return}\mathbin{\$}\Conid{MkDtry}\mathbin{\$}\Conid{Just}\;\Varid{d'}{}\<[E]%
\ColumnHook
\end{hscode}\resethooks

  With \ensuremath{\Varid{insert}} defined, \ensuremath{\Varid{fromPathMap}} is quite easy to define. We start with the empty directory, and progressively insert paths into it until there are no more paths left. If at any point, we try to insert a path that is a prefix of a pre-existing path, we fail.
\begin{hscode}\SaveRestoreHook
\column{B}{@{}>{\hspre}l<{\hspost}@{}}%
\column{3}{@{}>{\hspre}l<{\hspost}@{}}%
\column{14}{@{}>{\hspre}l<{\hspost}@{}}%
\column{29}{@{}>{\hspre}l<{\hspost}@{}}%
\column{32}{@{}>{\hspre}c<{\hspost}@{}}%
\column{32E}{@{}l@{}}%
\column{35}{@{}>{\hspre}l<{\hspost}@{}}%
\column{E}{@{}>{\hspre}l<{\hspost}@{}}%
\>[B]{}\Varid{insertPaths}\mathbin{::}[\mskip1.5mu ([\mskip1.5mu \Conid{Text}\mskip1.5mu],\Varid{a})\mskip1.5mu]\to \Conid{Dtry}\;\Varid{a}\to \Conid{Maybe}\;(\Conid{Dtry}\;\Varid{a}){}\<[E]%
\\
\>[B]{}\Varid{insertPaths}\;{}\<[14]%
\>[14]{}[\mskip1.5mu \mskip1.5mu]\;{}\<[29]%
\>[29]{}\Varid{d}{}\<[32]%
\>[32]{}\mathrel{=}{}\<[32E]%
\>[35]{}\Conid{Just}\;\Varid{d}{}\<[E]%
\\
\>[B]{}\Varid{insertPaths}\;{}\<[14]%
\>[14]{}((\Varid{p},\Varid{x})\mathbin{:}\Varid{pairs})\;{}\<[29]%
\>[29]{}\Varid{d}{}\<[32]%
\>[32]{}\mathrel{=}{}\<[32E]%
\>[35]{}\mathbf{do}{}\<[E]%
\\
\>[B]{}\hsindent{3}{}\<[3]%
\>[3]{}\Varid{d'}\leftarrow \Varid{insert}\;\Varid{p}\;\Varid{x}\;\Varid{d}{}\<[E]%
\\
\>[B]{}\hsindent{3}{}\<[3]%
\>[3]{}\Varid{insertPaths}\;\Varid{pairs}\;\Varid{d'}{}\<[E]%
\\[\blanklineskip]%
\>[B]{}\Varid{fromPathMap}\mathbin{::}\Dtry_0\;\Varid{a}\to \Conid{Maybe}\;(\Conid{Dtry}\;\Varid{a}){}\<[E]%
\\
\>[B]{}\Varid{fromPathMap}\;\Varid{d}\mathrel{=}\Varid{insertPaths}\;(\Varid{\Conid{Map}.toList}\;\Varid{d})\;(\Conid{MkDtry}\;\Conid{Nothing}){}\<[E]%
\ColumnHook
\end{hscode}\resethooks

  The wonderful thing about functional programming is that often by writing out your algorithm functionally, you can get a pretty decent induction argument. We will not do the proof that \ensuremath{\Varid{fromPathMap}} is an inverse in all of its fine detail, but a sketch looks like the following.

  Suppose that \ensuremath{\Varid{d}\mathbin{::}\Conid{Dtry}\;\Varid{a}}, \ensuremath{\Varid{d0}\mathbin{::}\Dtry_0\;\Varid{a}} are such that \ensuremath{\Varid{pathMap}\;\Varid{d}\equiv \Varid{d0}} and \ensuremath{\Varid{fromPathMap}\;\Varid{d0}\equiv \Conid{Just}\;\Varid{d}}. Then if \ensuremath{(\Varid{p},\Varid{x})\mathbin{::}([\mskip1.5mu \Conid{Text}\mskip1.5mu],\Varid{a})} is such that \ensuremath{\Varid{p}} is not a prefix of any key in \ensuremath{\Varid{d0}}, we have \ensuremath{\Varid{fromPathMap}\;(\Varid{\Conid{Map}.insert}\;\Varid{p}\;\Varid{x}\;\Varid{d0})\equiv \Conid{Just}\;(\Varid{insert}\;\Varid{p}\;\Varid{x}\;\Varid{d})} and \ensuremath{\Varid{pathMap}\;(\Varid{insert}\;\Varid{p}\;\Varid{x}\;\Varid{d})\equiv \Varid{\Conid{Map}.insert}\;\Varid{p}\;\Varid{x}\;\Varid{d0}}.

  Then, starting from the base case that \ensuremath{\Varid{pathMap}\;(\Conid{MkDtry}\;\Conid{Nothing})\equiv \Varid{\Conid{Map}.empty}} and \ensuremath{\Varid{fromPathMap}\;(\Varid{\Conid{Map}.empty})\equiv \Conid{Just}\;(\Conid{MkDtry}\;\Conid{Nothing})}, we can show inductively that if \ensuremath{\Varid{d0}\mathbin{::}\Dtry_0\;\Varid{a}} is produced by successively inserting new paths that are not prefixes of previous paths, then \ensuremath{\Varid{fmap}\;\Varid{pathMap}\;(\Varid{fromPathMap}\;\Varid{d0})\equiv \Conid{Just}\;\Varid{d0}}.

  Then by the fact \ensuremath{\Varid{pathMap}} is injective, we have that \ensuremath{\Varid{fromPathMap}} is an inverse when restriced to the collection of prefix-free \ensuremath{\Dtry_0}, as required.
\end{proof}

For the rest of the paper, we abbreviate the statement ``$U \Subset \SymbolSet^{\ast}$, $U$ is prefix-free'', by $U \PFSubset \SymbolSet^{\ast}$, so in particular we have $\Dtry \cong \sum_{U \PFSubset \SymbolSet^{\ast}} y^{U}$.

We can now understand the monad structure on $\Dtry$ in terms of this new presentation. We can write out the composition $\Dtry \circ \Dtry$ in the following way.
\[ \Dtry \circ \Dtry = \sum_{U \PFSubset \SymbolSet^{\ast}} \prod_{n \in U} \sum_{V_n \PFSubset \SymbolSet^{\ast}} y^{V_n} \cong \sum_{U \PFSubset \SymbolSet^{\ast}} \sum_{\{V_n \PFSubset \SymbolSet^{\ast}\}_{n \in U}} y^{\sum_{n \in U} V_n} \]
Then on positions, the monad multiplication $\mu \colon \Dtry \circ \Dtry \to \Dtry$ sends the $(U \PFSubset \SymbolSet^{\ast}, \{V_{n} \PFSubset \SymbolSet^{\ast}\}_{n \in U})$ to the prefix-free subset $U \ast \{V_{n}\}_{n \in U} = \{n \ast m \mid n \in U, m \in V_{n}\} \PFSubset \SymbolSet^{\ast}$, where $n \ast m$ is the concatentation of $n,m \in \SymbolSet^{\ast}$. On directions, we send each element $l \in U \ast \{V_{n}\}_{n \in U}$ to the unique pair $(n,m)$ such that $l = n \ast m$. The reason this works is that so long as $U$ and $V_{n}$ are prefix-free,

\[ U \ast \{V_{n}\}_{n \in U} \cong \sum_{n \in U} V_{n} \]

However, the crucial difference between $\sum_{n \in U} V_{n}$ and $U \ast \{V_{n}\}_{n \in U}$ is that concatentation is strictly associative and unital, while tuple-construction is not. That is, $n \ast (m \ast k) = (n \ast m) \ast k$, but $(n,(m,k)) \neq ((n,m),k)$. This will be relevant as a point of comparison for the next section.

\hide{
\begin{hscode}\SaveRestoreHook
\column{B}{@{}>{\hspre}l<{\hspost}@{}}%
\column{E}{@{}>{\hspre}l<{\hspost}@{}}%
\>[B]{}\mathbf{module}\;\Conid{Directories2}\;\mathbf{where}{}\<[E]%
\ColumnHook
\end{hscode}\resethooks
}

\section{Morphisms of directories}

Already, $\Dtry$ is useful as a monad on $\Set$. Specifically, if $T$ is some set of ``variable types,'' then an element of $\Dtry(T)$ represents a collection of typed, named variables. From an implementation standpoint, this is quite convenient. However, $\Dtry$ as a monad on $\Set$ does not tell a mathematical story that covers all of the operations we would like to perform on ``collections of typed, named variables.'' For instance, we might want to rename variables, or consider ``type-preserving functions'' between these collections of variables. The natural way to talk about this mathematically is to form a category where the elements of $\Dtry(T)$ are objects. Then we should lift the monad multiplication to a \emph{functor} $\Dtry(\Dtry(T)) \to \Dtry(T)$, so that we can compose type-preserving functions.

This is the motivation for lifting $\Dtry$ to a monad $\DtryCat$ on $\CatCat$, the category of categories. In fact, as we will see, there are several ways of doing this, depending on what sort of morphisms we want. In this section, we will cover these liftings and how they relate to some well-known constructions within category theory.

We will then prove that $\DtryCat$ is a \emph{2-monad}, which essentially means that it interacts with the 2-category structure of $\CatCat$ in a natural way. This leads to the consideration of 2-algebras of $\DtryCat$. A 2-algebra of $\DtryCat$ is a category $\cC$ with a functor $\DtryCat(\cC) \to \cC$ that can be seen as an ``unbiased monoidal product.'' That is, for any directory of objects in $\cC$, we can combine the objects into a single object. It will then turn out that the 2-category of 2-algebras of $\DtryCat$ is 2-equivalent to the 2-category of cocartesian monoidal categories. Other variations on $\DtryCat$ get the 2-categories equivalent to 2-category of cartesian monoidal categories or the 2-category of symmetric monoidal categories.

This solves a problem that comes up while implementing applied category theory on the computer. Namely, the classical presentation of (cartesian/cocartesian/symmetric) monoidal categories is done via binary monoidal products that are \emph{weakly associative}. For instance, coproduct in $\FinSet$ is only weakly associative. This is annoying because keeping track of the associators is painful. The alternative would be a strict monoidal product, which would require taking a skeleton of $\FinSet$. As an alternative, we can use directories, which give us the ability to name things in an intelligible way. A strict 2-algebra of $\DtryCat$ is mathematically equivalent to traditional monoidal categories, but allows strictness and human-meaningful names to work together.

\subsection{A note on implementation}

One could implement the constructions in this section, starting from a definition of categories in Haskell of the form

\begin{hscode}\SaveRestoreHook
\column{B}{@{}>{\hspre}l<{\hspost}@{}}%
\column{3}{@{}>{\hspre}l<{\hspost}@{}}%
\column{5}{@{}>{\hspre}l<{\hspost}@{}}%
\column{14}{@{}>{\hspre}c<{\hspost}@{}}%
\column{14E}{@{}l@{}}%
\column{18}{@{}>{\hspre}l<{\hspost}@{}}%
\column{E}{@{}>{\hspre}l<{\hspost}@{}}%
\>[3]{}\mathbf{data}\;\Conid{Category}\;\Varid{ob}\;\Varid{hom}\mathrel{=}\Conid{MkCategory}\;\{\mskip1.5mu {}\<[E]%
\\
\>[3]{}\hsindent{2}{}\<[5]%
\>[5]{}\Varid{dom}{}\<[14]%
\>[14]{}\mathbin{::}{}\<[14E]%
\>[18]{}\Varid{hom}\to \Varid{ob},{}\<[E]%
\\
\>[3]{}\hsindent{2}{}\<[5]%
\>[5]{}\Varid{codom}{}\<[14]%
\>[14]{}\mathbin{::}{}\<[14E]%
\>[18]{}\Varid{hom}\to \Varid{ob},{}\<[E]%
\\
\>[3]{}\hsindent{2}{}\<[5]%
\>[5]{}\Varid{id}{}\<[14]%
\>[14]{}\mathbin{::}{}\<[14E]%
\>[18]{}\Varid{ob}\to \Varid{hom},{}\<[E]%
\\
\>[3]{}\hsindent{2}{}\<[5]%
\>[5]{}\Varid{compose}{}\<[14]%
\>[14]{}\mathbin{::}{}\<[14E]%
\>[18]{}\Varid{hom}\to \Varid{hom}\to \Conid{Maybe}\;\Varid{hom}{}\<[E]%
\\
\>[3]{}\mskip1.5mu\}{}\<[E]%
\ColumnHook
\end{hscode}\resethooks

This follows the philosophy of computational category theory found in \cite{rydeheard_Computational_1988}, which is similar in spirit to work by the first author \cite{lynch_GATlab_2024}.

Howevever, this is not included in the current paper for several reasons. First of all, it is in fact somewhat awkward to work with categories in a programming language without dependent types. Secondly, we would have little use for a \emph{Haskell} implementation beyond the current paper; in the future we expect to implement this current section in Julia and perhaps Rust.

We now move on to the development of the theory for the 2-monad $\DtryCat$.

\subsection{The $\Fam$ construction} \label{sec:fam_construction}

The $\Fam$ construction is an endofunctor on $\CatCat$ that is similar to $\DtryCat(\cC)$. The $\Fam$-construction seems to be one of those results which, as is common in category theory, was known as folklore in the 60s and 70s and doesn't have a canonical original reference \cite{varkor_Original_2021}. Fortunately, unlike certain results which are left as folklore, it is not actually difficult to perform the $\Fam$ construction or verify its basic properties, and we do so now.

\begin{definition}
  Given a category $\cC$, a \textbf{family of objects} of $\cC$ consists of an \textbf{indexing set} $I$ and a choice of an object $X_{i} \colon \cC$ for each element $i \colon I$. We will often refer to a family of objects by the notation $(X_{i})_{i \colon I}$. \footnote{Note that we consider $(X_{i})_{i \colon I} = (X_{j})_{j \colon I}$, and we may use that renaming when convenient. In computer science, this is known as ``alpha-equivalence.''} A \textbf{morphism of families} from $(X_{i})_{i \colon I}$ to $(Y_{j})_{j \colon J}$ consists of a function $f \colon I \to J$ along with a family of morphisms $(f_{i} \colon X_{i} \to Y_{f(i)})_{i \colon I}$.
\end{definition}

\begin{example}
Pictured below is a morphism of families from $(X_i)_{i \colon \{1,\ldots,5\}}$ to $(Y_i)_{i \colon \{1,\ldots,4\}}$.

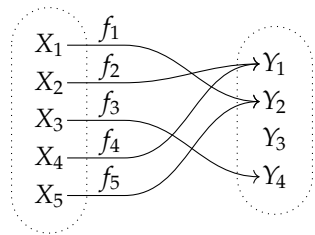
\begin{center}
\begin{tikzpicture}[every node/.style={outer sep=0, inner sep=1}, out=0, in=180]
  \node (X1) at (-2,     1) {$X_{1}$};
  \node (X2) at (-2,   0.5) {$X_{2}$};
  \node (X3) at (-2,     0) {$X_{3}$};
  \node (X4) at (-2,  -0.5) {$X_{4}$};
  \node (X5) at (-2,    -1) {$X_{5}$};
  \node (Y1) at ( 1,  0.75) {$Y_{1}$};
  \node (Y2) at ( 1,  0.25) {$Y_{2}$};
  \node (Y3) at ( 1, -0.25) {$Y_{3}$};
  \node (Y4) at ( 1, -0.75) {$Y_{4}$};
  \draw[dotted] (-2.5, 1) -- (-2.5, -1) arc [start angle=180, end angle = 360, radius=0.5] -- (-1.5, 1) arc [start angle=0, end angle = 180, radius=0.5];
  \draw[dotted] (0.5, 0.75) -- (0.5, -0.75) arc [start angle=180, end angle = 360, radius=0.5] -- (1.5, 0.75) arc [start angle=0, end angle = 180, radius=0.5];
  \draw[->] (X1) -- ($(X1)+(1, 0)$) node[pos=0.75, above] {$f_1$} to (Y2);
  \draw[->] (X2) -- ($(X2)+(1, 0)$) node[pos=0.75, above] {$f_2$} to (Y1);
  \draw[->] (X3) -- ($(X3)+(1, 0)$) node[pos=0.75, above] {$f_3$} to (Y4);
  \draw[->] (X4) -- ($(X4)+(1, 0)$) node[pos=0.75, above] {$f_4$} to (Y1);
  \draw[->] (X5) -- ($(X5)+(1, 0)$) node[pos=0.75, above] {$f_5$} to (Y2);
\end{tikzpicture}
\end{center}
\end{example}

\begin{definition}
  Given a category $\cC$, let $\Fam(\cC)$ be the category whose objects are families of objects of $\cC$ and whose morphisms are morphisms of families. Given families $(X_{i})_{i \colon I}$, $(Y_{j})_{j \colon J}$, $(Z_{k})_{k \colon K}$ and morphisms $(f_{i} \colon X_{i} \to Y_{f(i)})_{i \colon I}$, $(g_{j} \colon Y_{j} \to Z_{g(j)})_{j \colon J}$, we define the composite $(f \cmp g)$ by $(f \cmp g)(i) = g(f(i))$ and
  \[(f \cmp g)_{i} = X_{i} \xrightarrow{f_{i}} Y_{f(i)} \xrightarrow{g_{f(i)}} Z_{g(f(i))}\]
  Similarly, the identity on $(X_{i})_{i \colon I}$ consists of the identity on $I$, and the family $(1_{X_{i}})_{i \colon I}$.
\end{definition}

The category $\Fam(\cC)$ is sometimes also known as the \emph{free coproduct completion} of $\Fam(\cC)$. The intuition for this is that $(X_{i})_{i \colon I}$ as the ``formal coproduct'' of the $X_{i}$ in the following way.

Consider the functor $\Single \colon \cC \to \Fam(\cC)$ defined by $\Single(X) := (X)_{\_ \colon \{0\}}$, which is the family with indexing set $\{0\}$ and whose choice of object for $0$ is $X$. Then for an arbitrary family $(X_{i})_{i \colon I}$, for each $i \colon I$ there is an injection $\iota_{i} \colon \Single(X_{i}) \to (X_{i'})_{i' \colon I}$ where $\iota_{i}(0) = i$, and $\iota_{i,0} = 1_{X_{i}}$. These injections form a cocone, and it can be shown that this cocone satisfies the universal property making $(X_{i})_{i \colon I}$ the coproduct of the diagram $I \to \Fam(\cC)$, $i \mapsto \Single(X_{i})$.

What it means for $\Fam(\cC)$ to be ``free'' is that given a category $\cD$ with coproducts, there is a natural equivalence between the category of all functors $\cC \to \cD$ and the category of coproduct-preserving functors $\Fam(\cC) \to \cD$. Morally, $\Fam$ is a left adjoint to the forgetful functor from the category of ``categories with coproducts'' to the category of categories; but in a 2-categorical sense that we will not get into~\cite{perrone_Kan_2022}.

There are then several variations on the basic theme of the $\Fam$ construction, which produce different ``free completions.''

\begin{definition}
  For a category $\cC$, let $\FinFam(\cC)$ be the full subcategory\footnote{A full subcategory is a subcategory such that if $X$ and $Y$ are in the subcategory, then all morphisms between $X$ and $Y$ are also in the subcategory.} of $\Fam(\cC)$ consisting of families $(X_{i})_{i \colon I}$ with $I$ finite.
\end{definition}

The category $\FinFam(\cC)$ can be seen as the ``free finite coproduct completion'' of $\cC$, and satisfies a similar property to $\Fam(\cC)$ with respect to categories that have finite coproducts.

With a judicious application of duality, we can also get products. Specifically, for a category $\cC$, $\Fam(\cC^{\op})^{\op}$ is a category where objects are families $(X_{i})_{i \colon I}$, and a morphism $(X_{i})_{i \colon I} \to (Y_{j})_{j \colon J}$ consists of a function $f \colon J \to I$ along with a choice of $f_{j} \colon X_{f(j)} \to Y_{j}$ for each $j \colon J$.

We can now see how $\Fam(\cC^{\op})^{\op}$ is the ``free product completion'' of $\cC$. For a family $(X_{i})_{i \colon I}$, there are ``projection morphisms'' $\pi_{i} \colon (X_{i'})_{i' \colon I} \to \Single(X_{i})$ for each $i \colon I$, dual to the injection morphisms in $\Fam(\cC)$. Similarly, $\FinFam(\cC^{\op})^{\op}$ is the ``free finite product completion.''

Finally, we can do a similar construction for (possibly symmetric) (strict/weak) monoidal categories.

\begin{definition}
  For a category $\cC$, let $\FinFam_{\cong}(\cC)$ be the wide subcategory\footnote{A wide subcategory contains all of the objects, but not necessarily all of the morphisms} of $\FinFam(\cC)$ consisting of $f \colon (X_{i})_{i \colon I} \to (Y_{j})_{j \colon J}$ such that $i \mapsto f(i)$ is a bijection. Let $\FinFam_{=}(\cC)$ be the full subcategory of $\FinFam_{\cong}(\cC)$ consisting of only families on the indexing sets $I_{n} = \{1,\ldots,n\}$ for $n \colon \bN$. Let $\FinFam_{\cong}^{<}(\cC)$ be the category where the objects are pairs of \emph{totally ordered} indexing sets $I$ and families $(X_{i})_{i \colon I}$, and morphisms are order-preserving, and finally let $\FinFam_{=}^{<}(\cC)$ be the full subcategory of $\FinFam_{\cong}^{<}(\cC)$ on the totally ordered sets $I_{n}$.
\end{definition}

Respectively, $\FinFam_{\cong}(\cC)$, $\FinFam_{=}(\cC)$, $\FinFam_{\cong}^{<}(\cC)$, and $\FinFam_{=}^{<}(\cC)$ are the free symmetric weak monoidal category, free symmetric strict monoidal category, free weak monoidal category, and free strict monoidal category on $\cC$.

Any decent operation that acts on categories should also act on functors, and $\Fam$ is no exception.

\begin{definition}
  Suppose that $\cC$ and $\cD$ are two categories, and $F \colon \cC \to \cD$ is a functor between them. Then define $\Fam(F) \colon \Fam(\cC) \to \Fam(\cD)$ in the following way. For a family $(X_{i})_{i \colon I}$, let
  \[\Fam(F)((X_{i})_{i \colon I}) = (F(X_{i}))_{i \colon I}.\]
  Similarly, for a morphism $(f_{i} \colon X_{i} \to Y_{f(i)})_{i \colon I}$, let $\Fam(F)(f) = (F(f_{i}))_{i \colon I}$.
\end{definition}

It is not hard to show that with this definition $\Fam$ is indeed a functor $\CatCat \to \CatCat$, and mutatis mutandis all of the other $\Fam$ constructions ($\FinFam$, $\Fam((-)^{\op})^{\op}$, $\FinFam_{\cong}$, etc.) are likewise functorial.

It is much harder to show that $\Fam$ is a monad on $\CatCat$. This is because $\Fam$ is in fact \emph{not} a monad $\CatCat$, or at least not in the way we want it to be.

The obstruction is the following. We might want the monad multiplication for $\Fam$ to take a nested family $((X_{i,j})_{j \colon J_{i}})_{i \colon I}$ to the family $(X_{i,j})_{(i,j) \in \sum_{i \in I} J_{i}}$. However, when we have a triply-nested family $(((X_{i,j,k})_{k \colon K_{i,j}})_{j \colon J_{i}})_{i \colon I}$, then depending on the order in which we apply multiplication, we could end up with
\[ (X_{i,j,k})_{(i,(j,k)) \colon \sum_{i \colon I} \sum_{j \colon J_{i}} K_{i,j}}\]
or
\[ (X_{i,j,k})_{((i,j),k) \colon \sum_{(i,j) \colon \sum_{i \colon I} J_{i}} K_{i,j}}\]
which are isomorphic families, but not equal. For a fan of coherence conditions, it is a ``fun'' exercise to show that with this multiplication operation $\Fam$ forms a \emph{pseudomonad}, but fans of coherence conditions are few and far between.

Incidentally, $\FinFam_{=}$, which only takes its indexing sets to be $\{1,\ldots,n\}$, \emph{is} a monad (or better: a strict 2-monad), because in $\FinFam_{=}$ isomorphic indexing sets are equal. However, from a human-computer-interface perspective, indexing by integers is non-ideal. In the next section, we will show that we can do a construction equivalent to $\FinFam$ and $\FinFam_{=}$, which indexes families by paths in a directory, that has strict monad operations in the same way that $\FinFam_{=}$ does. This provides a coproduct completion that is user-friendly for the mathematician and programmer (via strictness) and is user-friendly for the modeler (via intelligble naming).



\subsection{Indexing by paths in a directory}

We now turn to the task of constructing $\DtryCat$. Just like with $\Fam$, there are many variants, but we will stick with the one that is analogous to $\FinFam$.

First, what precisely is $\DtryCat$? $\Dtry$ is a monad on $\Set$; one natural answer would be that $\DtryCat$ is a monad on $\CatCat$. However, $\CatCat$ has more structure than $\Set$; $\CatCat$ is a 2-category because $\CatCat$ has categories, functors, \emph{and} natural transformations. If we only define $\DtryCat$ as a monad, then it only applies to categories and functors. In order to make it apply to natural transformations, we must make it into a \emph{2-monad}. However, at a first pass the reader can ignore this complication, and just think about $\DtryCat$ as a monad on $\CatCat$; mentally replacing every time we say 2-monad with monad.

The basic idea is the following. There is a general construction detailed in \cref{appendix:lifting_monads} that allows us to take the monad $\Dtry$ on $\Set$ and lift it to a 2-monad $\DtryCat_{0}$ on $\CatCat$. However, this 2-monad is not equivalent to $\FinFam_{=}$; it ``doesn't have enough morphisms.'' We then ``graft'' the morphisms from $\FinFam_{=}$ onto $\DtryCat_{0}$ to make the final product $\DtryCat$; this uses another general construction detailed in \cref{appendix:boff_factorization}.

While we use some high-powered category theory in the appendices, this is only because we are lazy about proving things and thus want to use our favorite tools while doing so. The actual definitions of $\DtryCat_{0}$ and $\DtryCat$ are quite elementary.

\begin{definition}
  For a category $\cC$, which we can interpret (modulo size issues) as a span of sets $\cC_{0} \leftarrow \cC_{1} \to \cC_{0}$, let $\DtryCat_{0}(\cC)$ be the category with underlying span $\Dtry(\cC_{0}) \leftarrow \Dtry(\cC_{1}) \to \Dtry(\cC_{0})$.
\end{definition}

\begin{proposition}
  $\DtryCat_{0}$ is a 2-monad on $\CatCat$.
\end{proposition}

\begin{proof}
  See \cref{appendix:lifting_monads}.
\end{proof}

The problem with $\DtryCat_{0}(\cC)$ is there only exist morphisms between objects that have the same directory structures. That is, one can have a morphism in $\DtryCat_{0}(\Set)$ between $[\sym{a} \Rightarrow \bR, \sym{b.c} \Rightarrow \bZ]$ and $[\sym{a} \Rightarrow \bR^{2}, \sym{b.c} \Rightarrow \bB]$ given by something like $[\sym{a} \Rightarrow (x \mapsto (x,x)), \sym{b.c} \mapsto \mathrm{iseven}]$. However, there is no morphism between $[\sym{a} \Rightarrow \bR]$ and $[\sym{b} \Rightarrow \bR]$.

We solve this problem by ``adding in'' the morphisms from $\FinFam_{=}$, using the following construction.

\begin{proposition}
  Let $F \colon \cC \to \cD$ be a functor. Then there exists a unique-up-to-isomorphism factorization
  \[\begin{tikzcd}
      \cC && \cD \\
      & \cE
      \arrow["F", from=1-1, to=1-3]
      \arrow["{F_{\mathit{bo}}}"', from=1-1, to=2-2]
      \arrow["{F_{\mathit{ff}}}"', from=2-2, to=1-3]
    \end{tikzcd}\]
  such that $F_{\mathit{bo}}$ is bijective on objects, and $F_{\mathit{ff}}$ is fully-faithful. We call this the \textbf{bo-ff factorization} of $F$.
\end{proposition}

\begin{proof}
  This is well-known. The basic idea is to let $\cE_{0} = \cC_{0}$ and $\Hom_{\cE}(X, Y) = \Hom_{\cD}(F(X), F(Y))$.
\end{proof}

\begin{definition}
  Let $\cC$ be a category. Define $\PathFamily_{\cC} \colon \DtryCat_{0}(\cC) \to \FinFam(\cC)$ in the following way. Recall again that
  \[ \Dtry(\cC_{0}) = \sum_{d \colon \Dtry(1)} \cC_{0}^{\Dtry[d]} \]
  Accordingly, let
  \[ \PathFamily(d, X \colon \Dtry[d] \to \cC_{0}) = (X(p))_{p \colon \Dtry[d]} \colon \FinFam(\cC) \]
  Then define $\DtryCat(\cC)$ via the bo-ff factorization
  \[\begin{tikzcd}
    {\DtryCat_0(\cC)} && {\FinFam(\cC)} \\
    & {\DtryCat(\cC)}
    \arrow["\PathFamily_{\cC}", from=1-1, to=1-3]
    \arrow["{\PathFamily_{\mathit{bo}, \cC}}"', from=1-1, to=2-2]
    \arrow["{\PathFamily_{\mathit{ff}, \cC}}"', from=2-2, to=1-3]
  \end{tikzcd}\]
\end{definition}

If we write this out more explicitly, we get that $\DtryCat(\cC)$ is the category such that...

\begin{enumerate}
  \item An object is a pair $(d \colon \Dtry(1), X \colon \Dtry[d] \to \cC_{0})$
  \item A morphism from $(d, X)$ to $(d', X')$ consists of a pair of a function $f_{0} \colon \Dtry[d_{X}] \to \Dtry[d_{Y}]$ and for each $p \in \Dtry[d_{X}]$, a function $f_{1}(p) \colon X(p) \to X'(f_{0}(p))$. Equivalently, a morphism consists of a diagram
\[\begin{tikzcd}
	{\Dtry[d]} && {\Dtry[d']} \\
	& \cC
	\arrow["{f_0}", from=1-1, to=1-3]
	\arrow[""{name=0, anchor=center, inner sep=0}, "X"', from=1-1, to=2-2]
	\arrow[""{name=1, anchor=center, inner sep=0}, "{X'}", from=1-3, to=2-2]
	\arrow["{f_1}", shorten <=7pt, shorten >=7pt, Rightarrow, from=0, to=1]
\end{tikzcd}\]
\end{enumerate}
We show in \cref{appendix:boff_factorization} that with this definition, $\DtryCat$ is a 2-monad.
We conclude the story for $\DtryCat$ with the following proposition.

\begin{proposition}
  For a category $\cC$, the functor $\PathFamily_{\mathit{ff}, \cC} \colon \DtryCat(\cC) \to \FinFam(\cC)$ is an equivalence of categories.
\end{proposition}

\begin{proof}
  By definition, $\PathFamily_{\mathit{ff}}$ is fully faithful, so to show that it is an equivalence it suffices to show that it is essentially surjective. That is, for all $(Y_{i})_{i \colon I} \colon \FinFam(\cC)$, there exists $(d \colon \Dtry(1), X \colon \Dtry[d] \to \cC) \colon \DtryCat(\cC)$ with $\PathFamily_{\mathit{ff}} \cong (Y_{i})_{i \colon I}$. This follows from the fact that for any finite set $I$, there exists $d \colon \Dtry(1)$ such that $\Dtry[d] \cong I$, which can be proved by noting that one can make a binary tree with any number of leaves.
\end{proof}

Now, for concreteness we did all of the above starting from $\FinFam$. However, we learned in \cref{sec:fam_construction} that there are many variations of $\FinFam$. There are two variations of $\FinFam$ in particular that we are interested in: $\FinFam((-)^{\op})^{\op}$, and $\FinFam_{\cong}$. For the former, a morphism of $\DtryCat((-)^{\op})^{\op}$ from $(d,X)$ to $(d',X')$ consists of a diagram
\[\begin{tikzcd}
	{\Dtry[d]} && {\Dtry[d']} \\
	& \cC
	\arrow["{f_0}", from=1-1, to=1-3]
	\arrow[""{name=0, anchor=center, inner sep=0}, "X"', from=1-1, to=2-2]
	\arrow[""{name=1, anchor=center, inner sep=0}, "{X'}", from=1-3, to=2-2]
	\arrow["{f_1}"', shorten <=7pt, shorten >=7pt, Rightarrow, from=1, to=0]
\end{tikzcd}\]
For the latter, a morphism of $\DtryCat_{\cong}(-)$ consists of a diagram
\[\begin{tikzcd}
	{\Dtry[d]} && {\Dtry[d']} \\
	& \cC
	\arrow["{f_0}", from=1-1, to=1-3]
	\arrow[""{name=0, anchor=center, inner sep=0}, "X"', from=1-1, to=2-2]
	\arrow[""{name=1, anchor=center, inner sep=0}, "{X'}", from=1-3, to=2-2]
	\arrow["{f_1}", shorten <=7pt, shorten >=7pt, Rightarrow, from=0, to=1]
\end{tikzcd}\]
where $f_{0}$ is a bijection. We will see in the next section that 2-algebras of $\DtryCat$ correspond to cocartesian monoidal categories, 2-algebras of $\DtryCat((-)^{\op})^{\op}$ correspond to cartesian monoidal categories, and 2-algebras of $\DtryCat_{\cong}$ correspond to symmetric monoidal categories.

\subsection{2-algebras of $\DtryCat$}

In this section, we reap the rewards of the hard work we did proving that $\DtryCat$ was a 2-monad. The theory of 2-algebras of 2-monads is well-established, so there is essentially no new work in this section, but for the reader not well-versed in 2-monad theory we provide an account of some of the basics in the context of $\DtryCat$.

For the remainder of this section, let $T$ be one of $\DtryCat$, $\DtryCat((-)^{\op})^{\op}$, or $\DtryCat_{\cong}$; mostly we will think about $T = \DtryCat_{\cong}$ but the other cases are also useful. We now will essentially just follow \cite[4]{lack_2_2009}, explaining what each of the definitions means for these choices of $T$.

First, recall the definition of a 2-algebra. The reader familiar with regular monad algebras will notice that the definition is exactly the same; this is not a mistake. The distinction between algebras and 2-algebras comes only when we consider morphisms.

\begin{definition}
  A 2-monad 2-algebra for $T$ is a category $\cC$ along with a functor $A \colon T \cC \to \cC$ such that the following diagrams commute
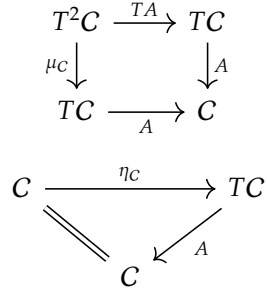
\[\begin{tikzcd}
	{T^2\cC} & {T\cC} \\
	{T\cC} & \cC
	\arrow["{TA}", from=1-1, to=1-2]
	\arrow["{\mu_\cC}"', from=1-1, to=2-1]
	\arrow["A", from=1-2, to=2-2]
	\arrow["A"', from=2-1, to=2-2]
\end{tikzcd}\]
\[\begin{tikzcd}
	\cC && {T\cC} \\
	& \cC
	\arrow["{\eta_\cC}", from=1-1, to=1-3]
	\arrow[Rightarrow, no head, from=1-1, to=2-2]
	\arrow["A", from=1-3, to=2-2]
\end{tikzcd}\]
  We will often just call $(\cC, A)$ a $T$-algebra.
\end{definition}

\begin{example}
  Any symmetric strict monoidal category $(\cC, \otimes, I)$ forms a $\DtryCat_{\cong}$-algebra. The map $A_{\otimes} \colon \DtryCat_{\cong}(\cC) \to \cC$ sends $(d \in \Dtry(1), X \colon \Dtry[d] \to \cC)$ to
  \[ \bigotimes_{p \in \Dtry[d]} X(p) \]
  where we order $\Dtry[d] \PFSubset \SymbolSet^{\ast}$ lexicographically. The functor $A_{\otimes}$ then sends a morphism in $\DtryCat_{\cong}(\cC)$ of the form
  
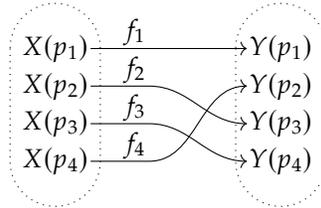
\begin{center}
  \begin{tikzpicture}[every node/.style={outer sep=0, inner sep=1}, out=0, in=180]
    \node (X1) at (-2,  0.75) {$X{(p_1)}$};
    \node (X2) at (-2,  0.25) {$X{(p_2)}$};
    \node (X3) at (-2, -0.25) {$X{(p_3)}$};
    \node (X4) at (-2, -0.75) {$X{(p_4)}$};
    \node (Y1) at ( 1,  0.75) {$Y{(p_1)}$};
    \node (Y2) at ( 1,  0.25) {$Y{(p_2)}$};
    \node (Y3) at ( 1, -0.25) {$Y{(p_3)}$};
    \node (Y4) at ( 1, -0.75) {$Y{(p_4)}$};
    \draw[dotted] (-2.6, 0.75) -- (-2.6, -0.75) arc [start angle=180, end angle = 360, radius=0.6] -- (-1.4, 0.75) arc [start angle=0, end angle = 180, radius=0.6];
    \draw[dotted] (0.4, 0.75) -- (0.4, -0.75) arc [start angle=180, end angle = 360, radius=0.6] -- (1.6, 0.75) arc [start angle=0, end angle = 180, radius=0.6];
    \draw[->] (X1) -- ($(X1)+(1.25, 0)$) node[pos=0.75, above] {$f_1$} to (Y1);
    \draw[->] (X2) -- ($(X2)+(1.25, 0)$) node[pos=0.75, above] {$f_2$} to (Y3);
    \draw[->] (X3) -- ($(X3)+(1.25, 0)$) node[pos=0.75, above] {$f_3$} to (Y4);
    \draw[->] (X4) -- ($(X4)+(1.25, 0)$) node[pos=0.75, above] {$f_4$} to (Y2);
  \end{tikzpicture}
  \end{center}
  to a morphism from $X(p_{1}) \otimes X(p_{2}) \otimes X(p_{3}) \otimes X(p_{4})$ to $Y(p_{1}) \otimes Y(p_{2}) \otimes Y(p_{3}) \otimes Y(p_{4})$ in $\cC$ represented by the following string diagram
  
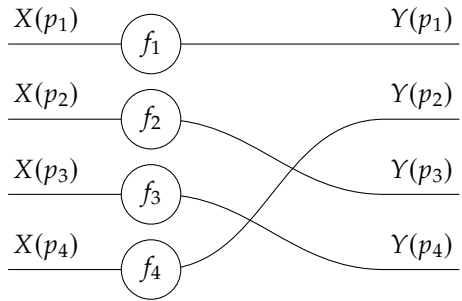
\begin{center}
  \begin{tikzpicture}[out=0, in=180]
    \tikzset{bead/.style={draw=black, fill=white, circle}}
    \coordinate (X1) at (-3,  1.5);
    \coordinate (X2) at (-3,  0.5);
    \coordinate (X3) at (-3, -0.5);
    \coordinate (X4) at (-3, -1.5);
    \coordinate (Y1) at ( 3,  1.5);
    \coordinate (Y2) at ( 3,  0.5);
    \coordinate (Y3) at ( 3, -0.5);
    \coordinate (Y4) at ( 3, -1.5);
    \draw (X1) -- ($(X1)+(1, 0)$) node[pos=0.5, above] {$X{(p_{1})}$} -- ($(X1)+(2, 0)$) node[pos=0.9, bead] {$f_{1}$} to ($(Y1)-(1,0)$) -- (Y1) node[pos=0.5, above] {${Y{(p_{1})}}$};
    \draw (X2) -- ($(X2)+(1, 0)$) node[pos=0.5, above] {$X{(p_{2})}$} -- ($(X2)+(2, 0)$) node[pos=0.9, bead] {$f_{2}$} to ($(Y3)-(1,0)$) -- (Y3) node[pos=0.5, above] {${Y{(p_{3})}}$};
    \draw (X3) -- ($(X3)+(1, 0)$) node[pos=0.5, above] {$X{(p_{3})}$} -- ($(X3)+(2, 0)$) node[pos=0.9, bead] {$f_{3}$} to ($(Y4)-(1,0)$) -- (Y4) node[pos=0.5, above] {${Y{(p_{4})}}$};
    \draw (X4) -- ($(X4)+(1, 0)$) node[pos=0.5, above] {$X{(p_{4})}$} -- ($(X4)+(2, 0)$) node[pos=0.9, bead] {$f_{4}$} to ($(Y2)-(1,0)$) -- (Y2) node[pos=0.5, above] {${Y{(p_{2})}}$};
  \end{tikzpicture}
  \end{center}
  One can think of $A_{\otimes}$ as an ``unbiased'' form of both the the monoidal product and the symmetries for a symmetric monoidal category .

  Similarly, any cartesian/cocartesian strict monoidal category forms an algebra of $\DtryCat((-)^{\op})^{\op}$/$\DtryCat$, respectively.
\end{example}

\begin{example}
  The ``tautological'' example of a $T$-algebra is the free $T$-algebra $\cC = T \cD, \mu_{\cD} \colon T\cC \to \cC$. Note that if $T = \DtryCat_{\cong}$ then $T\cC$ is a strict $T$-algebra, but it is \emph{not} a symmetric strict monoidal category! We have plenty of binary monoidal operations on $\cC$, because there is a functor $\cC \times \cC \to T \cC$ for every two-element directory. However, none of these binary monoidal operations are strictly associative. Instead, we have a different type of associativity, which comes from the square (where $A = \mu_{\cD}$).
\[\begin{tikzcd}
	{T^2\cC} & {T\cC} \\
	{T\cC} & \cC
	\arrow["{TA}", from=1-1, to=1-2]
	\arrow["{\mu_\cC}"', from=1-1, to=2-1]
	\arrow["A", from=1-2, to=2-2]
	\arrow["A"', from=2-1, to=2-2]
\end{tikzcd}\]
  From a mathematical perspective, this associativity is ``just as good'' as strict binary associativity. But while the objects of the free symmetric strict monoidal category on $\cD$ are lists of objects of $\cD$, the objects of $\DtryCat_{\cong}(\cD)$ are directories of objects of $\cD$, so rather than looking up an object by an integer index, we can look up objects by human-intelligible paths. Moreover, when we compose lists, we have to reindex everything, but when we compose directories, everything gets neatly put into subdirectories, so if the index before composition was $\sym{a.b}$, it might now be called $\sym{x.a.b}$, but the $\sym{a.b}$ part stays around.
\end{example}

In order for $\DtryCat_{\cong}$-algebras to be a reasonable replacement for symmetric (strict) monoidal categories, it should be the case that there is also a reasonable replacement for monoidal functors and monoidal natural transformations. But in fact ``monoidal functor'' is underspecified; there are a variety of different types of ``monoidal functor'' based on how we compare $F(X_{1} \otimes \cdots \otimes X_{n})$ with $F(X_{1}) \otimes \cdots \otimes F(X_{n})$. If they are equal, we have a ``strict monoidal functor'', if they are isomorphic, we have a ``weak monoidal functor'', if we have a morphism $F(X_{1}) \otimes \cdots \otimes F(X_{n}) \to F(X_{1} \otimes \cdots \otimes X_{n})$ then we have a ``lax monoidal functor'', and a morphism in the other direction is a ``oplax'' monoidal functor.

All of these have generalizations to $T$-algebras, and this is the motivation for the following definition.

\begin{definition}
  Let $(\cC, A \colon T \cC \to \cC)$ and $(\cD, B \colon T \cD \to \cD)$ be 2-algebras for $T$. Then given a functor $F \colon \cC \to \cD$, we can draw a square
\[\begin{tikzcd}
	{T \cC} & {T \cD} \\
	\cC & \cD
	\arrow["{T F}", from=1-1, to=1-2]
	\arrow["A"', from=1-1, to=2-1]
	\arrow["B", from=1-2, to=2-2]
	\arrow["F"', from=2-1, to=2-2]
\end{tikzcd}\]
  If we were working with a monad on a 1-category, we would require this square to commute, however, because we have a 2-monad in a 2-category, we have four choices on how to fill this square:
\[\begin{tikzcd}
	{T \cC} & {T \cD} & {T \cC} & {T \cD} \\
	\cC & \cD & \cC & \cD \\
	{T \cC} & {T \cD} & T\cC & T\cD \\
	\cC & \cD & \cC & \cD
	\arrow[""{name=0, anchor=center, inner sep=0}, "{T F}", from=1-1, to=1-2]
	\arrow["A"', from=1-1, to=2-1]
	\arrow["B", from=1-2, to=2-2]
	\arrow[""{name=1, anchor=center, inner sep=0}, "TF", from=1-3, to=1-4]
	\arrow["A"', from=1-3, to=2-3]
	\arrow["B", from=1-4, to=2-4]
	\arrow[""{name=2, anchor=center, inner sep=0}, "F"', from=2-1, to=2-2]
	\arrow[""{name=3, anchor=center, inner sep=0}, "F"', from=2-3, to=2-4]
	\arrow[""{name=4, anchor=center, inner sep=0}, "TF", from=3-1, to=3-2]
	\arrow["A"', from=3-1, to=4-1]
	\arrow["B", from=3-2, to=4-2]
	\arrow[""{name=5, anchor=center, inner sep=0}, "TF", from=3-3, to=3-4]
	\arrow["A"', from=3-3, to=4-3]
	\arrow["B", from=3-4, to=4-4]
	\arrow[""{name=6, anchor=center, inner sep=0}, "F"', from=4-1, to=4-2]
	\arrow[""{name=7, anchor=center, inner sep=0}, "F"', from=4-3, to=4-4]
	\arrow["{=}"{description}, draw=none, from=0, to=2]
	\arrow["\cong"{description}, draw=none, from=1, to=3]
	\arrow["{\bar{F}}", shorten <=6pt, shorten >=6pt, Rightarrow, from=4, to=6]
	\arrow["{\bar{F}}"', shorten <=6pt, shorten >=6pt, Rightarrow, from=7, to=5]
\end{tikzcd}\]

From left to right, top to bottom, these correspond to strict, pseudo, lax, and colax $T$-morphisms.

We also have additional coherence conditions which we will state only in the lax case; the other cases are analogous. Associativity for $(F, \bar{F})$ is the condition that the following are equal:
\[\begin{tikzcd}
	{T^2 \cC} & {T^2 \cD} && {T^2\cC} & {T^2 \cD} \\
	T\cC & T\cD & {=} & T\cC & T\cD \\
	\cC & \cD && \cC & \cD
	\arrow[""{name=0, anchor=center, inner sep=0}, "{T^2 F}", from=1-1, to=1-2]
	\arrow["{\mu_{\cC}}"', from=1-1, to=2-1]
	\arrow["{\mu_{\cD}}", from=1-2, to=2-2]
	\arrow[""{name=1, anchor=center, inner sep=0}, "{T^2 F}", from=1-4, to=1-5]
	\arrow["{T A}"', from=1-4, to=2-4]
	\arrow["TB", from=1-5, to=2-5]
	\arrow[""{name=2, anchor=center, inner sep=0}, "{T F}"{description}, from=2-1, to=2-2]
	\arrow["A"', from=2-1, to=3-1]
	\arrow["B", from=2-2, to=3-2]
	\arrow[""{name=3, anchor=center, inner sep=0}, "{T F}"{description}, from=2-4, to=2-5]
	\arrow["A"', from=2-4, to=3-4]
	\arrow["B", from=2-5, to=3-5]
	\arrow[""{name=4, anchor=center, inner sep=0}, "F"', from=3-1, to=3-2]
	\arrow[""{name=5, anchor=center, inner sep=0}, "F"', from=3-4, to=3-5]
	\arrow["{\mu_F}", shorten <=6pt, shorten >=6pt, Rightarrow, from=0, to=2]
	\arrow["{T \bar{F}}", shorten <=6pt, shorten >=6pt, Rightarrow, from=1, to=3]
	\arrow["{\bar{F}}", shorten <=6pt, shorten >=6pt, Rightarrow, from=2, to=4]
	\arrow["{\bar{F}}", shorten <=6pt, shorten >=6pt, Rightarrow, from=3, to=5]
\end{tikzcd}\]
Analogously, unitality for $(F, \bar{F})$ is the condition that the following are equal:
\[\begin{tikzcd}
	\cC & \cD && \cC & \cD \\
	T\cC & T\cD & {=} \\
	\cC & \cD && \cC & \cD
	\arrow[""{name=0, anchor=center, inner sep=0}, "F", from=1-1, to=1-2]
	\arrow["{\eta_\cC}"', from=1-1, to=2-1]
	\arrow["{\eta_\cD}", from=1-2, to=2-2]
	\arrow["F", from=1-4, to=1-5]
	\arrow[Rightarrow, no head, from=1-4, to=3-4]
	\arrow[Rightarrow, no head, from=1-5, to=3-5]
	\arrow[""{name=1, anchor=center, inner sep=0}, "TF"{description}, from=2-1, to=2-2]
	\arrow["A"', from=2-1, to=3-1]
	\arrow["B", from=2-2, to=3-2]
	\arrow[""{name=2, anchor=center, inner sep=0}, "F"', from=3-1, to=3-2]
	\arrow["F"', from=3-4, to=3-5]
	\arrow["{\eta_F}", shorten <=6pt, shorten >=6pt, Rightarrow, from=0, to=1]
	\arrow["{\bar{F}}", shorten <=6pt, shorten >=6pt, Rightarrow, from=1, to=2]
\end{tikzcd}\]
\end{definition}

To make sense of the coherence conditions, notice that in, for instance, the associativity coherence condition, the left side of the left diagram is $\mu_{C} \cmp A$ and the left side of the right diagram is $TA \cmp A$. The fact that these are equal is one of the prerequisites for $(\cC, A)$ to be an algebra for $T$. So these are very similar in spirit to associativity and unitality of algebras; in fact it turns out that lax morphisms are algebras for $T$ lifted to $\CatCat^{\to}$, the category where the objects are triples $(\cC, \cD, F \colon \cC \to \cD)$.

Anyways, the point is that when we apply the following definition to $\DtryCat_{\cong}$, we get good definitions for morphisms between the 2-algebras of $\DtryCat_{\cong}$.

Finally, we want an analogue of monoidal natural transformations.

\begin{definition}
  Suppose that $(F,\bar{F})$, $(G, \bar{G})$ are both lax $T$-morphisms from $(\cC, A)$ to $(\cD, B)$. Then a $T$-transformation from $F$ to $G$ consists of a natural transformation $\alpha \colon F \to G$ such that the following diagram commutes:

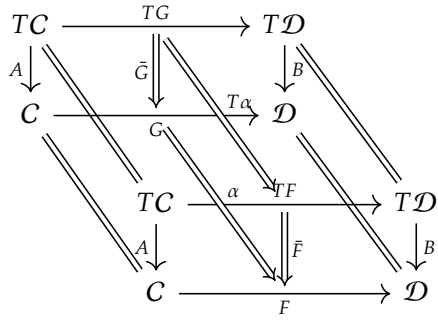
\[\begin{tikzcd}
	T\cC && T\cD \\
	\cC && \cD \\
	& T\cC && T\cD \\
	& \cC && \cD
	\arrow[""{name=0, anchor=center, inner sep=0}, "TG", from=1-1, to=1-3]
	\arrow["A"', from=1-1, to=2-1]
	\arrow[Rightarrow, no head, from=1-1, to=3-2]
	\arrow["B", from=1-3, to=2-3]
	\arrow[Rightarrow, no head, from=1-3, to=3-4]
	\arrow[""{name=1, anchor=center, inner sep=0}, "G"', from=2-1, to=2-3]
	\arrow[Rightarrow, no head, from=2-1, to=4-2]
	\arrow[Rightarrow, no head, from=2-3, to=4-4]
	\arrow[""{name=2, anchor=center, inner sep=0}, "TF", from=3-2, to=3-4]
	\arrow["A"', from=3-2, to=4-2]
	\arrow["B", from=3-4, to=4-4]
	\arrow[""{name=3, anchor=center, inner sep=0}, "F"', from=4-2, to=4-4]
	\arrow["{\bar{G}}"', shorten <=2pt, shorten >=2pt, Rightarrow, from=0, to=1]
	\arrow["{T \alpha}", shorten <=5pt, shorten >=5pt, Rightarrow, from=0, to=2]
	\arrow["\alpha", shorten <=5pt, shorten >=5pt, Rightarrow, from=1, to=3]
	\arrow["{\bar{F}}", shorten <=2pt, shorten >=2pt, Rightarrow, from=2, to=3]
\end{tikzcd}\]
\end{definition}

Using the above definitions, one can construct a 2-category of 2-algebras of $\DtryCat_{\cong}$, lax $\DtryCat_{\cong}$-morphisms and $\DtryCat_{\cong}$-transformations, which is analogous (and 2-equivalent) to the 2-category of symmetric monoidal categories, lax symmetric monoidal functors and monoidal natural transformations.

%

\hide{
  \begin{hscode}\SaveRestoreHook
\column{B}{@{}>{\hspre}l<{\hspost}@{}}%
\column{3}{@{}>{\hspre}l<{\hspost}@{}}%
\column{E}{@{}>{\hspre}l<{\hspost}@{}}%
\>[3]{}\mathbf{module}\;\Conid{Conclusion}\;\mathbf{where}{}\<[E]%
\ColumnHook
\end{hscode}\resethooks
}

\section{Conclusion}

Directories offer
a straightforward and efficient formalism
as well as
a practical data structure
for hierarchically-organized information.

The ideas presented in this paper
originated from our work on
Exergetic Port-Hamiltonian Systems (EPHS),
a modeling language for
interconnected macroscopic systems
\cite{lohmayer_Exergetic_2025}.
Starting from primitive subsystems,
mechanical, electromagnetic, and thermodynamic systems
can be hierarchically composed
using a graphical syntax
based on undirected wiring diagrams
\cite{2013Spivak}.
%
At the syntactic level,
a system is defined by its interface,
which consists of finitely many ports
through which energy is exchanged.
From a practical perspective,
naming systems and their ports
is essential for explicit reference.
For example,
a name like
\texttt{motor.stator.coil.magnetic\_flux},
indicates that
the \texttt{motor} system has
contains a subsystem called \texttt{stator},
which includes another subsystem \texttt{coil},
that has a port named \texttt{magnetic\_flux}.
%
These considerations have led to
the recognition of
the hierarchical `naming scheme'
used by EPHS
as a monad.

Besides providing a monadic data structure
that is instrumental for
the computer implementation of EPHS
\cite{2025Lohmayer},
directories provide an alternative to
symmetric monoidal categories (SMCs)
and their underlying multicategories.
Unlike SMCs,
directory-multicategories
have the advantage that they are inherently strict:
they use named tuples of subdirectories,
e.g.~ $(\mathtt{a} \mapsto x, \, \mathtt{b} \mapsto y, \, \mathtt{c} \mapsto z)$,
rather than nested binary tuples
(coming from a binary monoidal product),
e.g.~$(x, \, (y, \, z))$ or $((x, \, y), \, z))$,
which are only equivalent up to a coherence isomorphism.

Building on this foundation,
future work is well positioned to
formalize the graphical syntax
and relational semantics of EPHS
using the framework of
directory-multicategories and directory-functors.
This approach combines
mathematical rigor
with
practical concerns,
helping bridge the gap between theory and implementation.

\section*{Author contribution statement}%

\textbf{Owen Lynch}: Conceptualization, Investigation, Writing -- Original Draft, Writing -- Review \& Editing, Visualization;
\textbf{Markus Lohmayer}: Investigation, Writing -- Original Draft;


\bibliographystyle{link-elsarticle-num}
\addcontentsline{toc}{section}{References}
\bibliography{directories}

\appendix

\hide{
\begin{hscode}\SaveRestoreHook
\column{B}{@{}>{\hspre}l<{\hspost}@{}}%
\column{E}{@{}>{\hspre}l<{\hspost}@{}}%
\>[B]{}\mathbf{module}\;\Conid{FreeMonadDistrib}\;\mathbf{where}{}\<[E]%
\ColumnHook
\end{hscode}\resethooks
}

\section{Free Monads and Distributive Laws}

In this appendix, we complete the proof that there is a distributive law between $\freemonad_{\NERecord}$ and $\Maybe$.

One way to do this would be to translate the Haskell definition of \ensuremath{\Varid{distrib}} into math, and then check that four diagrams commute. However, there is a cleaner way which requires only that we translate \ensuremath{\Varid{filterNothings}} into math and check that two diagrams commute. We will prove this via some general categorical machinery for free monads. It may seem like this is a very long proof, but in fact once one understands some general facts about free monads and distributive laws, the proof is quite short. We simply take our time reviewing those general facts for the reader.

While it \emph{seems} fairly intuitive that $\Maybe \circ \freemonad_{\NERecord}$ is a monad, category theorists have had many experiences where things we hoped to be distributive laws were not in fact distributive laws, so it is worth being quite careful about this.

\begin{definition}
  If $F$ is an endofunctor on $\cC$, a \textbf{functor algebra} for $F$ is an object $X \in \cC$ along with a morphism $a \colon F(X) \to X$. We call $X$ the \textbf{carrier} and $a$ the \textbf{algebra structure on $X$}. A morphism of functor algebras from $(X,a)$ to $(Y, b)$ is a morphism $f \colon X \to Y$ such that the following square commutes.
  \[\begin{tikzcd}
    {F(X)} & {F(Y)} \\
    X & Y
    \arrow["{F(f)}", from=1-1, to=1-2]
    \arrow["a"', from=1-1, to=2-1]
    \arrow["b", from=1-2, to=2-2]
    \arrow["f"', from=2-1, to=2-2]
  \end{tikzcd}\]
  Let $\FAlg(F)$ be the category of functor algebras and functor algebra morphisms for $F$.
\end{definition}

\begin{definition}
  A monad algebra for a monad $T \colon \cC \to \cC$ consists of a functor algebra $(X, a)$ for $T$ such that the following diagrams commute
\[\begin{tikzcd}
	{T(T(X))} & {T(X)} \\
	{T(X)} & X
	\arrow["{\mu^T_X}", from=1-1, to=1-2]
	\arrow["{T(a)}"', from=1-1, to=2-1]
	\arrow["a", from=1-2, to=2-2]
	\arrow["a"', from=2-1, to=2-2]
\end{tikzcd}\]
\[\begin{tikzcd}
	& {T(X)} \\
	X && X
	\arrow["a", from=1-2, to=2-3]
	\arrow["{\eta_X^T}", from=2-1, to=1-2]
	\arrow[Rightarrow, no head, from=2-1, to=2-3]
\end{tikzcd}\]
  Let $\MAlg(T)$ be the full subcategory of $\FAlg(T)$ on monad algebras. $\MAlg(T)$ is also called the \textbf{Eilenberg-Moore} category for $T$ and written as $\cC^T$.
\end{definition}

Earlier we defined a free polynomial monad to be the result of applying the left adjoint to the forgetful functor from $\PolyMon$ to $\Poly$. If $\cC$ is an arbitrary category, $\Endo(\cC)$ is the category of endomorphisms on $\cC$, and $\Mon(\cC)$ is the category of monads on $\cC$, then there is in general not a left adjoint to the forgetful functor $\Mon(\cC) \to \Endo(\cC)$. However, there is still reasonable notion of ``the free monad'' on an endofunctor $F \colon \cC \to \cC$.

\begin{definition}
  For an endofunctor $F$, there is a forgetful functor $\Carrier_F \colon \FAlg(F) \to \cC$ which sends an algebra $(X,a)$ to its carrier $X$. If $\Carrier_F$ has a left adjoint, then we call that left adjoint $\Free_F \colon \cC \to \FAlg(F)$, and the compositte $\freemonad_F := \Carrier_F \circ \Free_F$ we call the \textbf{free monad} on $F$.
\end{definition}

\begin{lemma} \label{lemma:algebraically-free}
  If $P$ is a functor such that there is a free monad on $P$, and $\cC$ is complete and locally small, then $\FAlg(P)$ is isomorphic to $\MAlg(\freemonad_P)$.
\end{lemma}

\begin{proof}
  This can be found on the nlab \cite{nlabauthors_free_2024}, where the condition that $\FAlg(P)$ is isomorphic to $\MAlg(\freemonad_P)$ is termed ``algebraically-free.'' We will not recount this proof in full detail; rather we will just explain the construction going in either direction. First, if $a \colon \freemonad_P X \to X$ is a monad algebra of $\freemonad_P$, then we can make a $P$-algebra with $X$ as carrier via
  \[ PX \xrightarrow{P \eta_X} P \freemonad_P X \to \freemonad_P X \xrightarrow{a} X \]
  where the middle map is given by the fact that $\freemonad_P X$ is the free functor algebra for $P$ on $X$.

  In the other direction, given a functor algebra $b \colon P Y \to Y$, there is a natural map $\freemonad_P Y \to Y$ given by the arrow marked with $\ast$ in the following diagram that depicts the $\Free_P$/$\Carrier_P$ adjunction applied to the identity map $Y \to Y$ in $\cC$.
\[\begin{tikzcd}
	{P \freemonad_P Y} & {P Y} \\
	{\freemonad_P Y} & Y \\
	Y & Y
	\arrow[Rightarrow, no head, from=3-1, to=3-2]
	\arrow["{\Free_P}", maps to, from=3-1, to=2-1]
	\arrow["{\Carrier_P}", maps to, from=2-2, to=3-2]
	\arrow["\ast"', from=2-1, to=2-2]
	\arrow[from=1-1, to=2-1]
	\arrow["b", from=1-2, to=2-2]
	\arrow[from=1-1, to=1-2]
\end{tikzcd}\]
\end{proof}

We can now state the main theorem that we will use to show that there is a distributive law between $\freemonad_{\NERecord}$ and $\Maybe$.

\begin{theorem} \label{thm:half-distributive}
Suppose that $\cC$ is a complete and locally small category, $P$ is an endofunctor on $\cC$ with $\freemonad_P$ the free monad on $P$, and $(T,\eta^T,\mu^T)$ is any other monad. Also suppose that there is a natural transformation $\sigma \colon PT \to TP$ such that the following diagrams commute.
\[\begin{tikzcd}
	& {PT^2} \\
	TPT && PT \\
	{T^2P} && TP
	\arrow["{\sigma P}"', from=1-2, to=2-1]
	\arrow["{P\mu^T}", from=1-2, to=2-3]
	\arrow["P\sigma"', from=2-1, to=3-1]
	\arrow["\sigma", from=2-3, to=3-3]
	\arrow["{\mu^TP}"', from=3-1, to=3-3]
\end{tikzcd}\]
\[\begin{tikzcd}
	& P \\
	PT && TP
	\arrow["{P\eta^T}"', from=1-2, to=2-1]
	\arrow["{\eta^TP}", from=1-2, to=2-3]
	\arrow["\sigma"', from=2-1, to=2-3]
\end{tikzcd}\]
Then there is a distributive law $\freemonad_P T \to T \freemonad_P$.
\end{theorem}

To prove this, we start with an equivalent characterization of distributive laws.

\begin{definition}
  If $(T_1, \eta^{T_1}, \mu^{T_1})$ and $(T_2, \eta^{T_2}, \mu^{T_2})$ are both monads on a category $\cC$, a \textbf{lifting} of $T_2$ to $\MAlg(T_1)$ is a monad $(T_2^\ast, \eta^{T_2^\ast}, \mu^{T_2^\ast})$ on $\MAlg(T_1)$ that restricts to $(T_2, \eta^{T_2}, \mu^{T_2})$ on carriers.
\end{definition}

\begin{lemma} \label{lemma:distributive-laws-and-lifts}
  Distributive laws $T_1 T_2 \to T_2 T_1$ are in natural bijection with liftings of $T_2$ to $\MAlg(T_1)$.
\end{lemma}

\begin{proof}
  This is exercise 9.2 in \cite{barr_Toposes_2005}.
\end{proof}

\begin{lemma} \label{lemma:lift-to-falg}
  $T$ lifts to a monad on the category of functor algebras for $P$.
\end{lemma}

\begin{proof}
  First note that if $a \colon PX \to X$ is a functor algebra, then $PTX \xrightarrow{\sigma_X} TPX \xrightarrow{Ta} TX$ is also a functor algebra, because there are no conditions that we have to check. Moreover, if we have a morphism
\[\begin{tikzcd}
	PX & PY \\
	X & Y
	\arrow["Pf", from=1-1, to=1-2]
	\arrow["a"', from=1-1, to=2-1]
	\arrow["b", from=1-2, to=2-2]
	\arrow["f"', from=2-1, to=2-2]
\end{tikzcd}\]
  then the following commutes by naturality of $\sigma$ and functoriality of $T$.
\[\begin{tikzcd}
	PTX & PTY \\
	TPX & TPY \\
	TX & TY
	\arrow["{\sigma_X}"', from=1-1, to=2-1]
	\arrow["{\sigma_Y}", from=1-2, to=2-2]
	\arrow["{PT f}", from=1-1, to=1-2]
	\arrow["{TP f}", from=2-1, to=2-2]
	\arrow["Tf"', from=3-1, to=3-2]
	\arrow["Ta"', from=2-1, to=3-1]
	\arrow["Tb", from=2-2, to=3-2]
\end{tikzcd}\]
  Therefore, $a \mapsto \sigma \cmp Ta$ is functorial. It remains to show that $\eta^T$ and $\mu^T$ can be lifted to $\FAlg(P)$.

  For $\eta^T$, we must show that the following diagram commutes.
\[\begin{tikzcd}
	PX & PTX \\
	& TPX \\
	X & TX
	\arrow["Ta", from=2-2, to=3-2]
	\arrow["{\eta^T_X}"', from=3-1, to=3-2]
	\arrow["a"', from=1-1, to=3-1]
	\arrow["{\sigma_X}", from=1-2, to=2-2]
	\arrow["{P \eta_X^T}", from=1-1, to=1-2]
\end{tikzcd}\]
  This can be done by breaking it up in the following way
\[\begin{tikzcd}
	PX & PTX \\
	& TPX \\
	X & TX
	\arrow["Ta", from=2-2, to=3-2]
	\arrow["{\eta^T_X}"', from=3-1, to=3-2]
	\arrow["a"', from=1-1, to=3-1]
	\arrow["{\sigma_X}", from=1-2, to=2-2]
	\arrow["{P \eta_X^T}", from=1-1, to=1-2]
	\arrow["{\eta^T_{PX}}"', from=1-1, to=2-2]
\end{tikzcd}\]
The bottom square commutes by naturality of $\eta^T$, and the top triangle by the triangle-shaped condition in our assumptions.

  For $\mu^T$, we similarly show that the following diagram commutes, by naturality in the bottom square and then by our pentagon-shaped assumption for the top pentagon.
\[\begin{tikzcd}
	PTTX && PTX \\
	TPTX && TPX \\
	TTPX \\
	TTX && TX
	\arrow["Ta", from=2-3, to=4-3]
	\arrow["{P\mu_X^T}", from=1-1, to=1-3]
	\arrow["{\sigma_X}", from=1-3, to=2-3]
	\arrow["{\sigma_{TX}}"', from=1-1, to=2-1]
	\arrow["TTa"', from=3-1, to=4-1]
	\arrow["{\mu_X}"', from=4-1, to=4-3]
	\arrow["{T\sigma_X}"', from=2-1, to=3-1]
	\arrow["{\mu_{PX}}", from=3-1, to=2-3]
\end{tikzcd}\]

  We are done.
\end{proof}

The statement of \cref{thm:half-distributive} is a direct consequence of \cref{lemma:algebraically-free}, \cref{lemma:distributive-laws-and-lifts}, and \cref{lemma:lift-to-falg}, by the following argument.

\begin{enumerate}
  \item If we lift $T$ to a monad on $\MAlg(\freemonad_P)$, we get the distributive law we want (\cref{lemma:distributive-laws-and-lifts}).
  \item The category $\MAlg(\freemonad_P)$ is isomorphic to the category $\FAlg(P)$ (\cref{lemma:algebraically-free}).
  \item We can lift $T$ to a monad on $\FAlg(P)$ (\cref{lemma:lift-to-falg}).
\end{enumerate}

We now prove that the conditions of \cref{thm:half-distributive} hold for $P = \NERecord$ and $T = \Maybe$, with $\sigma = \mathit{filterNothings}$.

\begin{theorem} \label{thm:nerecord-maybe-distributive-law}
There is a distributive law between $\freemonad_{\NERecord}$ and $\Maybe$.
\end{theorem}

\begin{proof}
Because we have \cref{thm:half-distributive}, this is just a diagram chase for the following diagrams.

\[\begin{tikzcd}
	& {\NERecord \circ \Maybe \circ \Maybe} \\
	{\Maybe \circ \NERecord \circ \Maybe} && {\NERecord \circ \Maybe} \\
	{\Maybe \circ \Maybe \circ \NERecord} && {\Maybe \circ \NERecord}
	\arrow["{\NERecord \circ \mu^{\Maybe}}", from=1-2, to=2-3]
	\arrow["\filterNothings", from=2-3, to=3-3]
	\arrow["{\mu^{\Maybe} \circ \NERecord}"', from=3-1, to=3-3]
	\arrow[from=2-1, to=3-1]
	\arrow["{\filterNothings \circ \Maybe}"', from=1-2, to=2-1]
	\arrow["{\Maybe \circ \filterNothings}"', from=2-1, to=3-1]
\end{tikzcd}\]

\[\begin{tikzcd}
  & \NERecord \\
  {\NERecord \circ \Maybe} && {\Maybe \circ \NERecord}
  \arrow["\filterNothings"', from=2-1, to=2-3]
  \arrow["{\NERecord \circ \eta^{\Maybe}}"', from=1-2, to=2-1]
  \arrow["{\eta^{\Maybe} \circ \NERecord}", from=1-2, to=2-3]
\end{tikzcd}\]

The second diagam chase is more straightforwards. If we start with a non-empty record \ensuremath{\Varid{r}\mathbin{::}\mathit{Record}_{\neq \emptyset}\;\Varid{a}}s, apply \ensuremath{\Conid{Just}} to all of the elements of \ensuremath{\Varid{r}}, then filter out the \ensuremath{\Conid{Nothing}}s (the left-bottom path in the diagram) we get \ensuremath{\Conid{Just}\;\Varid{r}\mathbin{::}\Conid{Maybe}\;(\mathit{Record}_{\neq \emptyset}\;\Varid{a})} (the right path in the diagram).

For the first diagram chase, if we start out with \ensuremath{\Varid{r}\mathbin{::}\mathit{Record}_{\neq \emptyset}\;(\Conid{Maybe}\;(\Conid{Maybe}\;\Varid{a}))}, then we can consider two cases. The first case is that all of the elements of \ensuremath{\Varid{r}} are either \ensuremath{\Conid{Nothing}} or \ensuremath{\Conid{Just}\;\Conid{Nothing}}. Then either path will lead to \ensuremath{\Conid{Nothing}} in the end. Otherwise, we will end up with \ensuremath{\Conid{Just}\;\Varid{r'}}, where \ensuremath{\Varid{r'}} is the non-empty record consisting of \ensuremath{\Varid{x}} where \ensuremath{\Varid{r}} has an element of the form \ensuremath{\Conid{Just}\;(\Conid{Just}\;\Varid{x})}. We are done.
\end{proof}

\hide{
\begin{hscode}\SaveRestoreHook
\column{B}{@{}>{\hspre}l<{\hspost}@{}}%
\column{E}{@{}>{\hspre}l<{\hspost}@{}}%
\>[B]{}\mathbf{module}\;\Conid{LiftingMonads}\;\mathbf{where}{}\<[E]%
\ColumnHook
\end{hscode}\resethooks
}

\section{Lifting cartesian monads to monads on categories} \label{appendix:lifting_monads}

In this appendix, we prove the following theorem, which involves concepts to be defined later.

\begin{theorem}
  Let $(T,\eta^{T},\mu^{T})$ be a cartesian monad on a locally cartesian category $\cC$. Then there is a natural lift of $T$ to a 2-monad $\CatCat(T)$ on $\CatCat(\cC)$, the 2-category of categories internal to $\cC$, such that the action on the underlying spans of internal categories is given by sending $X_{0} \leftarrow X_{1} \to X_{0}$ to $TX_{0} \leftarrow TX_{1} \to TX_{0}$.
\end{theorem}

\begin{corollary}
  We can lift a cartesian monad on $\Set$ to a monad on $\CatCat$.
\end{corollary}

This result follows fairly directly from some general machinery, so while we are unsure whether or not this particular result appears in the literature, it is no doubt ``obvious'' to a sufficiently experienced category theorist. Likely there is also an ``elementary'' proof of this result, however, ``elementary'' proofs in 2-category theory are exceedingly tedious.

The reader interested in following this appendix should refer to the discussion of Cartesian monads in \cite{leinster_Higher_2004} and the monad construction in \cite{fiore_Monads_2011} when confused about the brief presentation that we give. Essentially, this is all an elaboration of the following concept, which goes back to \cite{street_formal_1972}.

\begin{definition}
  Let $\bC$ be a 2-category. A \textbf{monad in $\bC$} consists of an object $X \in \bC$ with an endomorphism $S \colon X \to X$, along with a pair of 2-cells $\eta \colon 1 \to S$, $\mu \colon S^{2} \to S$ such that the following diagrams commute:
\[\begin{tikzcd}
	S & {S^2} & S & {S^3} & {S^2} \\
	& S && {S^2} & S
	\arrow["{\eta S}", from=1-1, to=1-2]
	\arrow[Rightarrow, no head, from=1-1, to=2-2]
	\arrow["\mu", from=1-2, to=2-2]
	\arrow["{S \eta}"', from=1-3, to=1-2]
	\arrow[Rightarrow, no head, from=1-3, to=2-2]
	\arrow["{\mu S}", from=1-4, to=1-5]
	\arrow["{S \mu}"', from=1-4, to=2-4]
	\arrow["\mu", from=1-5, to=2-5]
	\arrow["\mu"', from=2-4, to=2-5]
\end{tikzcd}\]
\end{definition}

\begin{example}
  A monad in $\CatCat$ is a category with a monad on it.
\end{example}

\begin{example}
  A monad in the 2-category $\TwoCat$ of 2-categories, 2-functors and 2-natural transformations is a 2-monad.
\end{example}

\begin{definition}
  A \textbf{locally cartesian category} is a category which has all pullbacks. A \textbf{cartesian functor} is a functor between locally cartesian categories which preserves pullbacks. A \textbf{cartesian natural transformation} is a natural transformation $\alpha \colon F \Rightarrow G$ such that all of the naturality squares
\[\begin{tikzcd}
	Fx & Gx \\
	Fy & Gy
	\arrow["{\alpha_x}", from=1-1, to=1-2]
	\arrow["{F f}"', from=1-1, to=2-1]
	\arrow["{G f}", from=1-2, to=2-2]
	\arrow["{\alpha_y}"', from=2-1, to=2-2]
\end{tikzcd}\]
  are pullbacks.
\end{definition}

\begin{proposition}
  There is a 2-category $\LCC$ of locally cartesian categories, cartesian functors, and cartesian natural transformations.
\end{proposition}

\begin{definition}
  A \textbf{cartesian monad} is a monad in $\LCC$.
\end{definition}

\begin{theorem}
  If $F \colon \bC \to \bD$ is a 2-functor, and $(X,S,\eta,\mu)$ is a monad in $\bC$, then $(F(X), F(S), F(\eta), F(\mu))$ is a monad in $\bD$.
\end{theorem}

Our desired result will then follow if we can show that there is a 2-functor from $\LCC$ to $\TwoCat$ which sends a locally cartesian category $\sC$ to the 2-category of categories internal to $\sC$. We construct that 2-functor as the composite of 2-functors that can be found in the literature.
\[\begin{tikzcd}
	\LCC & \DblCat & \VirtDblCat & \TwoCat
	\arrow["\SpanCat", from=1-1, to=1-2]
	\arrow["\Mon", from=1-2, to=1-3]
	\arrow["\ArrowTwoCat", from=1-3, to=1-4]
\end{tikzcd}\]
In the above, $\DblCat$ is the 2-category of pseudo double categories, pseudo double functors, and tight natural transformations (natural transformations whose components are arrows, not proarrows). $\VirtDblCat$ is the 2-category of virtual double categories, virtual double functors, and tight natural transformations.

The 2-functor $\SpanCat$ takes a locally cartesian category $\cC$ to the double category where the objects are objects of $\cC$, the arrows are morphisms of $\cC$, the proarrows are spans in $\cC$, and the 2-cells are diagrams of the following shape
\[\begin{tikzcd}
	{A_1} & {X_1} & {B_1} \\
	{A_2} & {X_2} & {A_2}
	\arrow[from=1-1, to=2-1]
	\arrow[from=1-2, to=1-1]
	\arrow[from=1-2, to=1-3]
	\arrow[from=1-2, to=2-2]
	\arrow[from=1-3, to=2-3]
	\arrow[from=2-2, to=2-1]
	\arrow[from=2-2, to=2-3]
\end{tikzcd}\]
The fact that $\SpanCat$ is a 2-functor is a small modification of an exercise in \cite[3.7.5]{grandis_Higher_2019}.

$\Mon$ is the 2-functor from $\DblCat$ to $\VirtDblCat$ which takes a double category $\bD$ to the virtual double category of monads in $\bD$. More details on this can be found in \cite[5.3]{leinster_Higher_2004}, \cite[11]{shulman_Framed_2008}, or \cite[10]{lambert_Cartesian_2024}.

Finally, $\ArrowTwoCat$ is the 2-functor from $\VirtDblCat$ to $\TwoCat$ sending a virtual double category to the 2-category of objects, arrows and 2-cells with top and bottom given by identity proarrows.

Why does this composition of 2-functors get us what we want? Essentially, reasoning goes as follows. We first recall that internal categories in $\cC$ are promonads in $\SpanCat(\cC)$, the double category of spans in $\cC$. Let us unpack that statement.

An internal category in $\cC$ consists of a span $X_{0} \leftarrow X_{1} \to X_{0}$, along with span morphisms
\[\begin{tikzcd}
	{X_0} & {X_1} & {X_0} & {X_1} & {X_0} \\
	{X_0} && {X_1} && {X_0}
	\arrow[Rightarrow, no head, from=1-1, to=2-1]
	\arrow[from=1-2, to=1-1]
	\arrow[from=1-2, to=1-3]
	\arrow[shorten <=1pt, shorten >=1pt, Rightarrow, from=1-3, to=2-3]
	\arrow[from=1-4, to=1-3]
	\arrow[from=1-4, to=1-5]
	\arrow[Rightarrow, no head, from=1-5, to=2-5]
	\arrow[from=2-3, to=2-1]
	\arrow[from=2-3, to=2-5]
\end{tikzcd}\]
and
\[\begin{tikzcd}
	{X_0} & {X_0} & {X_0} \\
	{X_0} & {X_1} & {X_0}
	\arrow[Rightarrow, no head, from=1-1, to=2-1]
	\arrow[Rightarrow, no head, from=1-2, to=1-1]
	\arrow[Rightarrow, no head, from=1-2, to=1-3]
	\arrow[shorten <=2pt, shorten >=1pt, Rightarrow, from=1-2, to=2-2]
	\arrow[Rightarrow, no head, from=1-3, to=2-3]
	\arrow[from=2-2, to=2-1]
	\arrow[from=2-2, to=2-3]
\end{tikzcd}\]
which encode the composition and identity, respectively. This is precisely an object of $\Mon(\SpanCat(\cC))$, and moreover the 2-category of internal categories in $\cC$ is the precisely the arrow 2-category of $\Mon(\SpanCat(\cC))$. So we are done.

\hide{
\begin{hscode}\SaveRestoreHook
\column{B}{@{}>{\hspre}l<{\hspost}@{}}%
\column{E}{@{}>{\hspre}l<{\hspost}@{}}%
\>[B]{}\mathbf{module}\;\Conid{StrictificationViaBOFF}\;\mathbf{where}{}\<[E]%
\ColumnHook
\end{hscode}\resethooks
}

\section{Factorization of monads} \label{appendix:boff_factorization}

In this appendix, we prove the key lemma for the construction of the $\DtryCat$ 2-monad. Much of this proof depends on general facts about orthogonal factorization systems, which we then use in the specific case of the bo-ff factorization system mentioned in the text.

First of all, while we eventually want a 2-monad, we can use the following lemma in order to only deal with 1-category theory most of the time.

\begin{lemma} \label{lemma:unicity}
  If $F \colon \CatCat \to \CatCat$ is a 2-functor, let $F_{0}$ be the underlying functor. Then if $(F_{0}, \eta, \mu)$ forms a monad on $\CatCat$ viewed as a 1-category, $F$ is a 2-monad.
\end{lemma}

\begin{proof}
  This can be found in \cite[3.4]{power_Unicity_2009}.
\end{proof}

We start by recalling the definition of an orthogonal factorization system and some key facts about orthogonal factorization systems.

We have referred to \cite{joyal_Factorisation_2020} for an organized presentation of factorization systems, but the concept of a factorization system is well-known within category theory going back at least to the 1970s.

\begin{definition}
  An factorization system on a category $\cC$ is a pair $(\cL,\cR)$ of classes of maps in $\cC$ such that:
  \begin{enumerate}
    \item Every morphism $f \colon A \to B$ in $\cC$ admits a factorization $f = u \cmp p \colon A \to E \to B$, with $u \in \cL$ and $p \in \cR$.
    \item The classes $\cL$ and $\cR$ contain the isomorphisms and are closed under composition.
  \end{enumerate}
\end{definition}

Factorization systems have some nice properties beyond what is obvious from their definition. One of these properties is \emph{orthogonality}.

\begin{definition}
  Let $\cC$ be a category, and let $f \colon X \to Y$ and $g \colon Z \to W$ be maps in $\cC$. Then $f$ is left orthogonal to $g$ (equivalently, $g$ is right orthogonal to $f$) if for all $x \colon X \to Z$, $y \colon Y \to W$ such that the square commutes
\[\begin{tikzcd}
	X & Z \\
	Y & W
	\arrow["x", from=1-1, to=1-2]
	\arrow["f"', from=1-1, to=2-1]
	\arrow["g", from=1-2, to=2-2]
	\arrow["y"', from=2-1, to=2-2]
\end{tikzcd}\]
there exists a unique $u \colon Y \to Z$ that makes the following commute,
\[\begin{tikzcd}
    X & Z \\
    Y & W
    \arrow["x", from=1-1, to=1-2]
    \arrow["f"', from=1-1, to=2-1]
    \arrow["g", from=1-2, to=2-2]
    \arrow["u", from=2-1, to=1-2]
    \arrow["y"', from=2-1, to=2-2]
  \end{tikzcd}\]
\end{definition}

We then have the following fact about factorization systems.

\begin{proposition}
  Given a factorization system $(\cL, \cR)$, any morphism $f \in \cL$ is left orthogonal to any morphism $g \in cR$.
\end{proposition}

We can use this result to prove the following proposition.

\begin{proposition} \label{prop:factorize_functors}
  If $(\cL,\cR)$ is a factorization system on $\cC$, $\cD$ is any other category, and $F,G \colon \cD \to \cC$ are functors with $\alpha \colon F \Rightarrow G$ a natural transformation, then $\alpha$ factors as $\alpha = \upsilon \cmp \pi \colon F \Rightarrow H \Rightarrow G$, where $\upsilon_{d} \in \cL$ and $\pi_{d} \in \cR$ for all $d \colon \cD$.
\end{proposition}

\begin{proof}
  For each $d \colon \cD$, let $H(d)$, $\upsilon_{d}$, $\pi_{d}$ be given by a factorization
\[\begin{tikzcd}
	{F(d)} && {G(d)} \\
	& {H(d)}
	\arrow["{\alpha_d}", from=1-1, to=1-3]
	\arrow["{\upsilon_d}"', from=1-1, to=2-2]
	\arrow["{\pi_d}"', from=2-2, to=1-3]
\end{tikzcd}\]
It remains to define $H$ on morphisms and then to show that $\upsilon$ and $\pi$ are natural. We do this using the fact that for any $f \colon d \to d'$ in $\cD$, the morphism $\upsilon_{d}$ is left orthogonal to $\pi_{d'}$, so we may define $H(f)$ as the unique morphism making the following commute
\[\begin{tikzcd}
	{F(d)} & {F(d')} & {H(d')} \\
	\\
	{H(d)} & {G(d)} & {G(d')}
	\arrow["{F(f)}", from=1-1, to=1-2]
	\arrow["{\upsilon_d}"', from=1-1, to=3-1]
	\arrow["{\upsilon_{d'}}", from=1-2, to=1-3]
	\arrow["{\pi_{d'}}", from=1-3, to=3-3]
	\arrow["{H(f)}", from=3-1, to=1-3]
	\arrow["{\pi_{d'}}"', from=3-1, to=3-2]
	\arrow["{G(f)}"', from=3-2, to=3-3]
\end{tikzcd}\]
Note that the above diagram actually also contains the naturality squares for $\upsilon$ and $\pi$, so we are done.
\end{proof}

This extends to the following proposition.

\begin{proposition}
  Given a factorization system $(\cL, \cR)$ on a category $\cC$, and given any other category $\cD$, there is a factorization system $(\cL^{\cD}, \cR^{\cD})$ on $\cC^{\cD}$ where $\cL^{\cD}$ consists of natural transformations whose components are in $\cL$ and $\cR^{\cD}$ consists of natural transformations whose components are in $\cR$.
\end{proposition}

Finally, we need one more technical result before we can prove our main theorem.

\begin{proposition}
  Let $(\cL, \cR)$ be a factorization system on $\cC$, and let $F,G \colon \cC \to \cC$ be endofunctors such that at least one of $F$ and $G$ preserves the factorization system. Then for $\alpha \colon F \Rightarrow G$ in $\cL^{\cC}$, $\alpha^{2} \colon F^{2} \Rightarrow G^{2}$ is also in $\cL^{\cC}$.
\end{proposition}

\begin{proof}
  This follows immediately by the definition of the components of $\alpha^{2}$
\[\begin{tikzcd}
	{F^2(c)} & {F(G(c))} \\
	{G(F(c))} & {G^2(c)}
	\arrow["{F(\alpha_c)}", from=1-1, to=1-2]
	\arrow["{\alpha_{F(c)}}"', from=1-1, to=2-1]
	\arrow["{\alpha^2_c}"{description}, from=1-1, to=2-2]
	\arrow["{\alpha_{G(c)}}", from=1-2, to=2-2]
	\arrow["{G(\alpha_c)}"', from=2-1, to=2-2]
\end{tikzcd}\]
 and the fact that $\cL$ is closed under composition.
\end{proof}

We now reach our main technical lemma.

\begin{lemma} \label{lemma:monad_factorization}
  Let $\cC$ be a category with a factorization system $(\cL, \cR)$. Suppose that $T$ and $S$ are monads on $\cC$ that each preserve the factorization system, and $\alpha \colon T \to S$ is a monadic natural transformation, i.e. the diagram

\[\begin{tikzcd}
	{T^2} & {S^2} \\
	T & S \\
	1 & 1
	\arrow["{\alpha^2}", from=1-1, to=1-2]
	\arrow["{\mu^T}"', from=1-1, to=2-1]
	\arrow["{\mu^S}", from=1-2, to=2-2]
	\arrow["\alpha"', from=2-1, to=2-2]
	\arrow["{\eta^T}", from=3-1, to=2-1]
	\arrow[Rightarrow, no head, from=3-1, to=3-2]
	\arrow["{\eta^S}"', from=3-2, to=2-2]
\end{tikzcd}\]
commutes. Then if $\upsilon \cmp \pi \colon T \Rightarrow R \Rightarrow S$ is the factorization of $\alpha$ as given in \cref{prop:factorize_functors}, then there is a monad structure $\mu^{R}, \eta^{R}$ on $R$ such that the following commutes:
\[\begin{tikzcd}
	{T^2} & R^{2} & {S^2} \\
	T & R & S \\
	1 & 1 & 1
	\arrow["{\upsilon^2}", from=1-1, to=1-2]
	\arrow["{\mu^T}"', from=1-1, to=2-1]
	\arrow["{\pi^2}", from=1-2, to=1-3]
	\arrow["{\mu^R}"', from=1-2, to=2-2]
	\arrow["{\mu^S}", from=1-3, to=2-3]
	\arrow["\upsilon"', from=2-1, to=2-2]
	\arrow["\pi"', from=2-2, to=2-3]
	\arrow["{\eta^T}", from=3-1, to=2-1]
	\arrow[Rightarrow, no head, from=3-1, to=3-2]
	\arrow["{\eta^R}", from=3-2, to=2-2]
	\arrow[Rightarrow, no head, from=3-2, to=3-3]
	\arrow["{\eta^S}"', from=3-3, to=2-3]
\end{tikzcd}\]
\end{lemma}

\begin{proof}
  It is simple to define $\eta^{R} \colon 1 \to R$ by $\eta^{R} = \eta^{T}\cmp \upsilon$; this makes the bottom rectangle commute by definition. For $\mu^{R}$, we invoke orthogonality of $\upsilon^{2}$ and $\pi$ to get:
\[\begin{tikzcd}
	{T^2} & T & R \\
	\\
	{R^2} & {S^2} & S
	\arrow["{\mu^T}", from=1-1, to=1-2]
	\arrow["{\upsilon^2}"', from=1-1, to=3-1]
	\arrow["\upsilon", from=1-2, to=1-3]
	\arrow["\pi", from=1-3, to=3-3]
	\arrow["{\mu^R}", from=3-1, to=1-3]
	\arrow["{\pi^2}"', from=3-1, to=3-2]
	\arrow["{\mu^S}"', from=3-2, to=3-3]
\end{tikzcd}\]
We now must show that $(R,\eta^{R},\mu^{R})$ satisfies the monad laws. We will show associativity; unitality is easier and left to the reader. Consider the following diagram
\begin{equation} \label{eq:factorization_cube}
\begin{tikzcd}
	& {T^3} & {R^3} & {S^3} \\
	& {T^2} & {R^2} & {S^2} \\
	{T^2} &&& {R^2} && {S^2} \\
	\\
	T &&& R && S
	\arrow["{\upsilon^3}", from=1-2, to=1-3]
	\arrow["{T\mu^T}"{description}, from=1-2, to=2-2]
	\arrow["{\mu^T T}"{description}, from=1-2, to=3-1]
	\arrow["{\pi^3}", from=1-3, to=1-4]
	\arrow["{R\mu^R}"{description}, from=1-3, to=2-3]
	\arrow["{\mu^R R}"{description}, from=1-3, to=3-4]
	\arrow["{S\mu^S}"{description}, from=1-4, to=2-4]
	\arrow["{\mu^S S}"{description}, from=1-4, to=3-6]
	\arrow["{\upsilon^2}"', from=2-2, to=2-3]
	\arrow["{\mu^T}"{description}, from=2-2, to=5-1]
	\arrow["{\pi^2}"', from=2-3, to=2-4]
	\arrow["{\mu^R}"{description}, from=2-3, to=5-4]
	\arrow["{\mu^S}"{description}, from=2-4, to=5-6]
	\arrow["{\upsilon^2}", from=3-1, to=3-4]
	\arrow["{\mu^T}"', from=3-1, to=5-1]
	\arrow["{\pi^2}", from=3-4, to=3-6]
	\arrow["{\mu^R}", from=3-4, to=5-4]
	\arrow["{\mu^S}", from=3-6, to=5-6]
	\arrow["{\upsilon}", from=5-1, to=5-4]
	\arrow["{\pi}", from=5-4, to=5-6]
  \end{tikzcd}
\end{equation}
The left and right squares are associativity for $T$ and $S$. The front squares and bottom rectangles squares commute by how we defined $R$ and $\mu^{R}$. Finally, the top and back squares commute by general arguments in the monoidal category of endofunctors. Now, consider the following commutative square.
\begin{equation} \label{eq:assoc_filler}
\begin{tikzcd}
	{T^3} & R \\
	{R^3} & S
	\arrow[from=1-1, to=1-2]
	\arrow["\upsilon^{3}"', from=1-1, to=2-1]
	\arrow["\pi", from=1-2, to=2-2]
	\arrow[from=2-1, to=2-2]
  \end{tikzcd}
\end{equation}
In (\ref{eq:assoc_filler}), the top morphism is any of the equal ways of getting from $T^{3}$ to $R$ in (\ref{eq:factorization_cube}) without going through $R^{3}$, and the bottom morphism is any of the equal ways of getting from $R^{3}$ to $S$ without going through $R$.

By the commutativity of the various squares in (\ref{eq:factorization_cube}), both $R \mu^{R} \cmp \mu^{R}$ and $\mu^{R} R \cmp \mu^{R}$ are fillers of (\ref{eq:assoc_filler}). Therefore, by uniqueness of fillers, we have $R \mu^{R} \cmp \mu^{R} = \mu^{R} R \cmp \mu^{R}$, as required for associativity. We are done.
\end{proof}

We can now prove the result cited in the text, that $\DtryCat$ is a 2-monad.

\begin{theorem}
  The assignment $\cC \mapsto \DtryCat(\cC)$ forms the action on objects of a 2-monad on $\CatCat$.
\end{theorem}

\begin{proof}
  Most of the heavy lifting is done by \cref{lemma:monad_factorization}. However, there are a couple loose ends to tie up.

  First of all, the definition of $\DtryCat$ in the text uses a factorization of $\PathFamily \colon \DtryCat_{0} \to \FinFam$. However, $\FinFam$ is not a 2-monad; it is a pseudomonad. Thus, in order to apply \cref{lemma:monad_factorization}, we must instead use a map $\DtryCat_{0} \to \FinFam_{=}$. We build this in the following way. First note that there is not a reasonable functor $\FinSet \to \FinSet_{=}$, where $\FinSet_{=}$ is the category with objects the sets $\{1,\ldots,n\}$ for $n \in \bN$ and with morphisms as functions. This is because such a functor would be tantamount to choosing a linear order on every finite set. However, there is a natural functor $\FinSet_{>} \to \FinSet_{=}$, where $\FinSet_{>}$ is the category of finite sets with chosen linear orders, though the morphisms are not necessarily order-preserving. Similarly, we can build a functor $\FinFam_{>}$ on $\CatCat$, which is analogous to $\FinFam$ except the indexing families are equipped with a linear order, and there is a natural transformation $\FinFam_{>} \to \FinFam_{=}$.

  We can exploit this fact because if we fix a linear order on $\SymbolSet$ this then induces a linear lexicographic order on $\Dtry[d] \PFSubset \SymbolSet^{\ast}$ for any $d \in \Dtry(1)$. Thus, the natural transformation $\DtryCat_{0} \to \FinFam$ in fact factors through $\DtryCat_{0} \to \FinFam_{>}$. We then compose this with the map $\FinFam_{>}$ to $\FinFam_{=}$. That this natural transformation is monadic is a consequence of the fact that the monad multiplication $\Dtry \circ \Dtry \to \Dtry$ is order-preserving for the lexicographic order.

  Secondly, we must show that $\DtryCat$ is a 2-functor. Suppose that $\cC$ and $\cD$ are categories, and $F,G \colon \cC \to \cD$ are functors, with $\alpha \colon F \Rightarrow G$ a natural transformation. Then $\DtryCat_{0}(\alpha)$ is a natural transformation from $\DtryCat_{0}(F)$ to $\DtryCat_{0}(G)$, which sends an object $(d \in \Dtry(1), X \colon \Dtry[d] \to \cC_{0})$ to the morphism $(d, X \cmp \alpha)$ from $(d, X \cmp F_{0})$ to $(d, X \cmp G_{0})$. We can then push this forward along the inclusion $\DtryCat_{0}(\cD) \to \DtryCat(\cD)$ to get a (possibly natural) transformation $\DtryCat(\alpha) \colon \DtryCat(F) \Rightarrow \DtryCat(G)$. We must check that it is still natural with respect to the morphisms in $\DtryCat(\cC)$ though.

  Let $(d \in \Dtry(1), X \colon \Dtry[d] \to \cC_{0})$ and $(d' \in \Dtry(1), X' \colon \Dtry[d'] \to \cC_{0})$ be objects of $\DtryCat(\cC)$, and consider a morphism
  \[\left(f_{0} \colon \Dtry[d] \to \Dtry[d'], f_{1} \colon \prod_{p \in \Dtry[d]} \Hom_{\cC}(X(p), X'(f_{1}(p)))\right)\]
  which can be pictured as
\[\begin{tikzcd}
	{\Dtry[d]} && {\Dtry[d']} \\
	& \cC
	\arrow["{f_0}", from=1-1, to=1-3]
	\arrow[""{name=0, anchor=center, inner sep=0}, "X"', from=1-1, to=2-2]
	\arrow[""{name=1, anchor=center, inner sep=0}, "{X'}", from=1-3, to=2-2]
	\arrow["{f_1}", shorten <=7pt, shorten >=7pt, Rightarrow, from=0, to=1]
\end{tikzcd}\]
  Then we must show that the following square commutes
\[\begin{tikzcd}
	{\DtryCat(F)((d,X))} && {\DtryCat(F)((d',X'))} \\
	{\DtryCat(G)((d,X))} && {\DtryCat(G)((d',X'))}
	\arrow["{\DtryCat(F)(f_0,f_1)}", from=1-1, to=1-3]
	\arrow["{\DtryCat(\alpha)_{(d,X)}}"', from=1-1, to=2-1]
	\arrow["{\DtryCat(\alpha)_{(d',X')}}", from=1-3, to=2-3]
	\arrow["{\DtryCat(G)(f_0,f_1)}"', from=2-1, to=2-3]
\end{tikzcd}\]
This square expands out to
\[\begin{tikzcd}
	{\Dtry[d]} &&&& {\Dtry[d']} \\
	\\
	&& \cD \\
	\\
	{\Dtry[d]} &&&& {\Dtry[d']}
	\arrow["{f_0}", from=1-1, to=1-5]
	\arrow[""{name=0, anchor=center, inner sep=0}, "{X\cmp F}"{description}, from=1-1, to=3-3]
	\arrow[Rightarrow, no head, from=1-1, to=5-1]
	\arrow[""{name=1, anchor=center, inner sep=0}, "{X' \cmp F}"{description}, from=1-5, to=3-3]
	\arrow[Rightarrow, no head, from=1-5, to=5-5]
	\arrow[""{name=2, anchor=center, inner sep=0}, "{X \cmp G}"{description}, from=5-1, to=3-3]
	\arrow["{f_0}"', from=5-1, to=5-5]
	\arrow[""{name=3, anchor=center, inner sep=0}, "{X' \cmp G}"{description}, from=5-5, to=3-3]
	\arrow["{f_1 \cmp F}", shorten <=14pt, shorten >=14pt, Rightarrow, from=0, to=1]
	\arrow["{X \cmp \alpha}"', shorten <=8pt, shorten >=8pt, Rightarrow, from=0, to=2]
	\arrow["{X' \cmp \alpha}", shorten <=8pt, shorten >=8pt, Rightarrow, from=1, to=3]
	\arrow["{f_1 \cmp G}"', shorten <=14pt, shorten >=14pt, Rightarrow, from=2, to=3]
\end{tikzcd}\]
The outer square clearly commutes, and for each element $p \in \Dtry[d]$, the inner square boils down to
\[\begin{tikzcd}
	{F(X(p))} & {F(X'(f_0(p)))} \\
	{G(X(p))} & {G(X'(f_0(p)))}
	\arrow["{F(f_1(p))}", from=1-1, to=1-2]
	\arrow["{\alpha_{X(p)}}"', from=1-1, to=2-1]
	\arrow["{\alpha_{X'(f_0(p))}}", from=1-2, to=2-2]
	\arrow["{G(f_1(p))}"', from=2-1, to=2-2]
\end{tikzcd}\]
which commutes by naturality of $\alpha$. We have thus checked that $\DtryCat$ is a 2-functor, and thus a 2-monad by \cref{lemma:unicity}. We are done.
\end{proof}


\end{document}

%% file: preamble.tex
\usepackage{subcaption}
\usepackage{microtype}

\usepackage{amsmath}
\usepackage{amssymb}

\usepackage{tikz}
\usepackage{quiver}
\usepackage{adjustbox} 

\usepackage{cancel}

\usepackage{siunitx}

\usepackage{physics}

\usepackage{mathtools}

\usepackage{scalerel}

\usepackage{stmaryrd}

\usepackage{listings}
\lstdefinelanguage{HaskellMin}%
  {otherkeywords={=>},%
   morekeywords={if,then,else,case,class,data,deriving,hiding,in,infix,infixl,infixr,import,instance,let,module,newtype,of,qualified,type,where,do,as,family,default,forall,foreign,proc,rec},%
   sensitive,
   morecomment=[l]--,%
   morecomment=[n]{\{-}{-\}},%
   morestring=[b]"%
 }[keywords,comments,strings]

 \usepackage{dirtree}

\DeclareMathOperator{\Hom}{Hom}

\newcommand{\bB}{\mathbb{B}}
\newcommand{\bC}{\mathbb{C}}
\newcommand{\bD}{\mathbb{D}}

\newcommand{\bN}{\mathbb{N}}

\newcommand{\bR}{\mathbb{R}}

\newcommand{\bZ}{\mathbb{Z}}

\newcommand{\cC}{\mathcal{C}}
\newcommand{\cD}{\mathcal{D}}
\newcommand{\cE}{\mathcal{E}}

\newcommand{\cL}{\mathcal{L}}

\newcommand{\cR}{\mathcal{R}}

\newcommand{\sC}{\mathscr{C}}

\newcommand{\gcat}[1]{\mathsf{#1}}

\newcommand{\gset}[1]{\mathit{#1}}

\renewcommand{\op}{\mathrm{op}}

\newcommand{\FinSet}{\gcat{FinSet}}

\newcommand{\Set}{\gcat{Set}}

\newcommand{\SpanCat}{\gcat{Span}}

\newcommand{\CatCat}{\gcat{Cat}}
\newcommand{\Fam}{\gcat{Fam}}
\newcommand{\Single}{\gset{Single}}

\newcommand{\FAlg}{\gcat{FAlg}}
\newcommand{\MAlg}{\gcat{MAlg}}
\newcommand{\Endo}{\gcat{Endo}}
\newcommand{\Mon}{\gcat{Mon}}
\newcommand{\Carrier}{\gset{Carrier}}
\newcommand{\Free}{\gset{Free}}
\newcommand{\LCC}{\gcat{LCC}}
\newcommand{\DblCat}{\gcat{DblCat}}
\newcommand{\VirtDblCat}{\gcat{VirtDblCat}}
\newcommand{\TwoCat}{\text{2-}\gcat{Cat}}
\newcommand{\ArrowTwoCat}{\gcat{Arrow}\text{2-}\gcat{Cat}}

\newcommand{\into}{\hookrightarrow}

\newcommand{\cmp}{\fatsemi}

\newcommand{\SymbolSet}{\gset{Sym}}

\newcommand{\Dtry}{\gset{Dtry}}
\newcommand{\DtryCat}{\gcat{Dtry}}

\newcommand{\List}{\gset{List}}
\newcommand{\Record}{\gset{Record}}
\newcommand{\NERecord}{\gset{Record}_{\neq \emptyset}}
\newcommand{\Maybe}{\gset{Maybe}}
\newcommand{\filterNothings}{\mathit{filterNothings}}
\newcommand{\pMap}{\mathit{pathMap}}
\newcommand{\PFSubset}{\Subset_{\mathit{pf}}}
\newcommand{\PathFamily}{\mathit{PathFamily}}

\newcommand{\Poly}{\gcat{Poly}}
\newcommand{\PolyMon}{\gcat{PolyMon}}

\newcommand{\colim}{\varinjlim}

\newcommand{\freemonad}{\mathfrak{m}}

\newcommand{\FinFam}{\gcat{FinFam}}

\newcommand{\sym}[1]{\mathtt{#1}}

\usepackage[unicode=true]{hyperref}

\hypersetup{%
	plainpages=false,
	bookmarksopen=true,
	bookmarksnumbered=true,
	breaklinks=true,
}

\usepackage{amsthm}

\usepackage[capitalize,nameinlink,noabbrev]{cleveref}

\newtheoremstyle{break} 
  {}          
  {}          
  {}          
  {}          
  {\bfseries} 
  {.}         
  {\newline}  
  {}          
\theoremstyle{break}

\newtheorem{theoremx}{Theorem}[section]
\newenvironment{theorem}
  {\pushQED{\qed}\theoremx}
  {\popQED\endtheoremx}

\newtheorem{propositionx}[theoremx]{Proposition}
\newenvironment{proposition}
  {\pushQED{\qed}\propositionx}
  {\popQED\endpropositionx}

\newtheorem{corollaryx}[theoremx]{Corollary}
\newenvironment{corollary}
  {\pushQED{\qed}\corollaryx}
  {\popQED\endcorollaryx}

\newtheorem{definitionx}[theoremx]{Definition}
\newenvironment{definition}
  {\pushQED{\qed}\definitionx}
  {\popQED\enddefinitionx}

\newtheorem{lemmax}[theoremx]{Lemma}
\newenvironment{lemma}
  {\pushQED{\qed}\lemmax}
  {\popQED\endlemmax}

\newtheorem{examplex}[theoremx]{Example}
\newenvironment{example}
  {\pushQED{\qed}\examplex}
  {\popQED\endexamplex}

\newtheorem*{remarkxx}{Remark}
\newenvironment{remark*}
  {\pushQED{\qed}\remarkxx}
  {\popQED\endremarkxx}

\usepackage{booktabs} 



%% file: abstract.tex
This paper introduces an inherently strict presentation of
categories with products, coproducts, or symmetric monoidal products
that is inspired by
file systems and directories.
Rather than using nested binary tuples
to combine objects or morphisms,
the presentation uses named tuples.
Specifically, we develop 2-monads whose strict 2-algebras are
product categories, coproduct categories, or symmetric monoidal categories,
in a similar vein to the classical Fam construction,
but where the elements of the indexing set are
period-separated identifiers like $\sym{cart.motor.momentum}$.
Our development of directories is also intended to serve
the secondary purpose of expositing certain aspects of polynomial monads,
and is accompanied by Haskell code
that shows how the mathematical ideas can be implemented.